\documentclass{article}
\usepackage{amsmath}
\usepackage{amssymb}
\usepackage{amsfonts}
\usepackage[top=2cm,bottom=2cm,left=2.5cm,right=2cm]{geometry}

\usepackage{amsmath}
\usepackage{graphicx}
\usepackage[colorinlistoftodos]{todonotes}
\usepackage[colorlinks=true, allcolors=blue]{hyperref}
\usepackage{amsthm}
\usepackage{bbm}
\usepackage{enumerate}
\usepackage{verbatim}
\usepackage{biblatex}
\usepackage{float}
\usepackage[affil-it]{authblk}

\newtheorem{theorem}{Theorem}[section]
\newtheorem{claim}[theorem]{Claim}

\newtheorem{lemma}[theorem]{Lemma}
\newtheorem{proposition}[theorem]{Proposition}
\newtheorem{corollary}[theorem]{Corollary}
\newtheorem{definition}[theorem]{Definition}

\begin{document}

\title{Pointwise bounds for Eisenstein series on $\Gamma_0(q)\char`\\SL_2(\mathbb{R})$}
\author[1]{Evgeny Musicantov\thanks{musicantov.evgeny@gmail.com}}
\author[2]{Sa'ar Zehavi\thanks{saarzehavi@gmail.com}}
\affil[1]{The Hebrew University of Jerusalem, Jerusalem, Israel}
\affil[2]{Tel Aviv University, Tel Aviv, Israel}

\maketitle

\begin{abstract}
    We construct pointwise bounds in the weight aspect for Eisenstein series on $X_0(q) = \Gamma_0(q)\char`\\SL_2(\mathbb{R})$, with squarefree level $q$, using a Sobolev technique. More specifically, we show that for an Eisenstein series $E$ on $X_0(q)$ of weight parameter $n$ and type $t$, one has for all $x\in X_0(q)$: $|E(x,1/2 + it)| \ll_{\epsilon} q^{\epsilon}(1 + |n|^{1/2 + \epsilon} + |t|^{1/2 + \epsilon})\sqrt{y(x) + y(x)^{-1}}$, where $y(x)$ is the Iwasawa $y$-coordinate of the point $x$.
\end{abstract}

\section{Introduction}
Let $\Omega$ be a compact, smooth Riemannian manifold, $\Delta$ the Laplace operator, and $f$ a Laplace eigenfunction, i.e. there exists some $\lambda$ such that
\[
-\Delta f = \lambda f.
\]
Since $f$ is smooth (by elliptic theory) and $\Omega$ is compact, $f$ is bounded. Finding an upper bound on the maximal value of $f$ on $\Omega$ in terms of $|\lambda|$ is a classical problem. In quantum mechanics, such upper bounds are a measure for the equidistribution of mass of a high energy particle, which is related to the phenomenon of quantum scarring.

For non-compact manifolds, Laplace eigenfunctions no longer necessarily attain a maximum absolute value. An analogue of the previous problem is studied by the name of the sup-norm problem, which typically asks for an upper bound on the supremum of a function on our manifold, when normalized in a certain way. Other than the quantum mechanical interpretation, the sup-norm problem has connections to the multiplicity problem, zero sets and nodal lines of automorphic functions, and bounds for Faltings’ delta function. See~\cite{RUDNICK},~\cite{MORA},~\cite{NODAL} and~\cite{JORGENSON}.

Our focus in this paper is on a special family of non-compact manifolds, called modular manifolds and on certain Laplace eigenfunctions on such manifolds called Eisenstein series. Modular manifolds are obtained by taking the left quotient of $SL_2(\mathbb{R})$ by a congruence subgroup $\Gamma_0(q)$, or by taking its double quotient by a congruence subgroup from the left, and by a maximal compact subgroup, such as $SO_2(\mathbb{R})$, from the right. The quotient space $SL_2(\mathbb{R})/SO_2(\mathbb{R})$ is called the hyperbolic plane, $\mathbb{H} = \{x + iy\in \mathbb{C}:y > 0\}$, and is equipped with the hyperbolic metric $ds^2 = \dfrac{dx^2 + dy^2}{s^2}$.

The Laplacian on $\mathbb{H}$ is
\[
\Delta_{\mathbb{H}} = y^2(\dfrac{d^2}{dx^2} + \dfrac{d^2}{dy^2}).
\]
Similarly, $SL_2(\mathbb{R})$ is parametrized by the set $\{(x,y,\theta)\in \mathbb{R}\times \mathbb{R}_+\times \mathbb{S}^1\}$. 

Given an element $g \in SL_2(\mathbb{R})$, $g$ can be decomposed uniquely into a product $g = nak$, where 
\[
n\in N =
\left\{
\begin{pmatrix}
1& x\\
0& 1
\end{pmatrix}: x\in\mathbb{R}
\right\},\quad
a\in A =
\left\{
\begin{pmatrix}
y^{1/2}& 0\\
0& y^{-1/2}
\end{pmatrix}: y\in\mathbb{R}_+
\right\},\quad
k\in SO_2(\mathbb{R}) =
\left\{
\begin{pmatrix}
\cos\theta& -\sin\theta\\
\sin\theta& \cos\theta
\end{pmatrix}: \theta\in[0,2\pi]
\right\}.
\]
Given an element $g = nak$, we call $nak$ the Iwasawa decomposition of $g$, and the corresponding $(x,y,\theta)$ are called the Iwasawa coordinates of $g$.

On $SL_2(\mathbb{R})$, we have the Casimir operator, denoted by $\mathfrak{C}$, which is given, in Iwasawa coordinates, by
\[
\mathfrak{C} = -y^2(\dfrac{d^2}{dx^2} + \dfrac{d^2}{dy^2}) + y\dfrac{d}{dx}\dfrac{d}{d\theta}.
\]
There is also a notion of Laplacian on $SL_2(\mathbb{R})$, which is induced by the choice of a metric. For the standard metric $ds^2 = \dfrac{dx^2 + dy^2 + d\theta^2}{y^2}$, the Laplacian is related to the Casimir by
\[
\Delta = -\mathfrak{C} + \dfrac{1}{2}\dfrac{d^2}{d\theta^2}.
\]

Smooth functions on $\Gamma_0(q)\char`\\\mathbb{H}$ lift naturally into smooth functions on $\Gamma_0(q)\char`\\SL_2(\mathbb{R})$ via the pullback of the projection morphism $\pi:\Gamma_0(q)\char`\\SL_2(\mathbb{R})\longrightarrow \Gamma_0(q)\char`\\\mathbb{H}$. Since such functions are independent of the Iwasawa $\theta$ coordinate, the action of the Casimir and the action of the Laplacian coincide on such functions. 

\textbf{Remark}: smooth functions on $\Gamma_0(q)\char`\\SL_2(\mathbb{R})$ independent of $\theta$ are called \textit{spherical}.

\textbf{Remark}: the previous observation implies that the spectrum of the Laplacian on $\Gamma_0(q)\char`\\\mathbb{H}$ injects into the (spherical) spectrum of the Casimir on $\Gamma_0(q)\char`\\SL_2(\mathbb{R})$. 

Both the spectrum of the Casimir on $\Gamma_0(q)\char`\\ SL_2(\mathbb{R})$ and the Laplacian on $\Gamma_0(q)\char`\\ \mathbb{H}$ consist of two types: the discrete and continuous spectrum. The discrete spectrum is comprised of the constant function(s) and cusp forms. The Riemannian 3-fold $\Gamma_0(q)\char`\\ SL_2(\mathbb{R})$ and the 2-fold $\Gamma_0(q)\char`\\ \mathbb{H}$ have cusps, and cusp forms are functions that vanish on those cusps. Maass-cusp forms are square integrable and bounded functions. For a Maass-cusp form $f$ on the modular manifold $X$, we define its sup-norm by 
\[
||f||_{\infty} = \sup_{x\in X}\dfrac{|f(x)|}{||f||_2}.
\]
The supnorm problem for Hecke-Maass cusp newforms on arithmetic quotients of the upper half plane has been widely studied in the literature. Throughout the discussion, let $u_t$ denote such a Hecke-Maass cusp newforms of type $t$ on $\Gamma_0(q)\char`\\\mathbb{H}$, that is with Laplace eigenvalue of $\lambda = 1/4 + t^2$.

Iwaniec and Sarnak (see~\cite{IWASAR}) proved for the level 1 case and Blomer-Holowinsky (see~\cite{BLOMERHOLO}) for the general level case, that for any $\epsilon > 0$,
\[
||u_t||_{\infty} \ll_{q,\epsilon} |t|^{5/12 + \epsilon}.
\]
Blomer and Holowinsky also supplied the first non-trivial bound in the level aspect. For squarefree level $q$
\[
||u_t||_{\infty} \ll_{t} q^{-1/37}.
\]
This bound was subsequently improved by Templier~\cite{TEMP1} and by Templier-Harcos in~\cite{TEMP-HARC1} and~\cite{TEMP-HARC2} to
\[
||u_t||_{\infty} \ll_{t,\epsilon} q^{-1/6 + \epsilon},\quad \forall \epsilon > 0.
\]
It is conjectured that the optimal bound is $||u_t||_{\infty} \ll_{t,\epsilon} q^{-1/2 + \epsilon}$. In~\cite{TEMP2}, Templier obtains the best known hybrid bound, i.e. bound in both level and spectral aspects simultaneously, for squarefree level $q$
\[
||u_t||_{\infty} \ll_{\epsilon} |t|^{5/12 + \epsilon}q^{-1/6 + \epsilon},\quad \forall \epsilon > 0.
\]
The focus of this paper is on the Eisenstein series of $\Gamma_0(q)\char`\\SL_2(\mathbb{R})$ with large weight:

\begin{definition}[Weight $n$ forms]
Let $n\in\mathbb{Z}$. A weight $n$ automorphic form on the modular 3-fold $X_0(q)$, are functions $f:X_0(q)\longrightarrow\mathbb{C}$, with the property that for all 
\[
k = 
\begin{pmatrix}
\cos\theta& -\sin\theta\\
\sin\theta& \cos\theta
\end{pmatrix}\in SO_2(\mathbb{R}),
\]
and all $g\in SL_2(\mathbb{R})$:
\[
f(gk) = e^{in\theta}f(g).
\]
\end{definition}

\textbf{Remark}: we denote the modular 3-fold $\Gamma_0(q)\char`\\SL_2(\mathbb{R})$ by $X_0(q)$ all throughout the paper.

\textbf{Remark}: not all Casimir eigenforms on $X_0(q)$ are weight $n$ forms for any $n\in\mathbb{Z}$. Generally, the sum of two or more forms of different weight will give such an example.

\textbf{Remark}: we shall refer to automorphic forms of integer weight as functions of ``pure-weight".

\textbf{Remark}: spherical functions are weight-0 functions.

Bounds on the absolute values of the pointwise evaluation of Maass cusp forms in the weight aspect have been studied by Bernstein and Reznikov, see~\cite{BERREZ}. Denote by $u_{t,n}(x)$ a Maass cusp form of weight $n$, type $t$ and $L^2$-norm equal to 1. In~\cite[p. 2166]{BERREZ}, Bernstein and Reznikov prove
\[
\forall x\in X_0(q):\quad |u_{t,n}(x)| \ll_{q,\epsilon} (1 + |t|^{1/2 + \epsilon} + |n|^{1/2 + \epsilon})\sqrt{y(x) + y(x)^{-1}},
\]
where $y(x)$ is the Iwasawa $y$-coordinate of $x$. The above bound implies, in particular, that $|||u_{t,n}||_{\infty} \ll_{q,\epsilon} (1 + |t|^{1/2 + \epsilon} + |n|^{1/2 + \epsilon})$.

We introduce our notations. Let $\mathfrak{A}$ be the set of cusps of the modular 3-fold $X_0(q) = \Gamma_0(q)\char`\\SL_2(\mathbb{R})$. We denote by $E_{\mathfrak{a},n}(*,1/2+it)$ the weight $n$ Eisenstein series of the cusp $\mathfrak{a}\in\mathfrak{A}$, with spectral parameter $s = 1/2 + it$, $t\in\mathbb{R}$. The Eisenstein series above, evaluated at a point $x\in X_0(q)$ is written as $E_{\mathfrak{a},n}(x,1/2+it)$, and with abuse of notation, we will consider the same Eisenstein series as a function on $SL_2(\mathbb{R})$ in the standard way (via the pullback of the projection morphism).

The analogue of the sup-norm problem for the weight $0$ (or ``spherical") Eisenstein series was studied by Young, in~\cite[Theorem 1.1]{YOUNG}: Given the unique Eisenstein series $E(z,1/2 + it) = E_{\infty,0}(z,1/2+it)$ of the quotient $SL_2(\mathbb{Z})\char`\\\mathbb{H}$, normalized by the constant term, let $\Omega\subseteq \mathbb{H}$ be any compact subdomain of $\mathbb{H}$, assuming $t \ge 1$ then for any $\epsilon > 0$
\[
\max_{z\in\Omega}|E(z,1/2 + it)| \ll_{\Omega,\epsilon} t^{3/8 + \epsilon}.
\]
At this time, the state of the art is the result of Huang and Xu, who showed~\cite{HUANGXU}.
\begin{theorem}[Huang \& Xu,~\cite{HUANGXU}, Theorem 1.2]
\label{theorem:huang_xu}
Let $z = x + iy\in\mathbb{H}$, $q$ a positive squarefree integer, and $\mathfrak{a}$ a cusp of $\Gamma_0(q)\char`\\\mathbb{H}$. For $t\ge 1$, denote by $E_{\mathfrak{a}}(*,1/2 + it)$ the Eisenstein series of the cusp $\mathfrak{a}$ with spectral parameter $1/2 + it$, normalized by the constant term, then for all $\epsilon > 0$
\[
E_{\mathfrak{a}}(z,1/2 + it) = \delta_{\mathfrak{a}=\infty}y^{1/2 + it} + \phi_{\mathfrak{a},\infty}(1/2 + it)y^{1/2 - it} + O_{\epsilon}(q^{-1/2 + \epsilon}(y^{-1/2} + t^{3/8 + \epsilon})).
\]
\end{theorem}

Although Huang and Xu's dependence on $q$ seems optimal, and gives an improvement in the spectral aspect (similar to Young's), their theorem says nothing about the non-spherical case (non-zero weight).

We will show:
\begin{theorem}[Main Theorem]
\label{theorem:main_theorem}
Let $g\in SL_2(\mathbb{R})$ have Iwasawa coordinates $x,y,\theta$, let $q$ be a positive squarefree integer, $n\in 2\mathbb{Z}$, and $\mathfrak{a}$ be a cusp of $X_0(q)$. For $t\in\mathbb{R}$, denote by $E_{\mathfrak{a},n}(*,1/2 + it)$ the Eisenstein series of the cusp $\mathfrak{a}$, weight $n$, with spectral parameter $1/2 + it$, normalized by the constant term. Then
\[
|E_{\mathfrak{a},n}(g,1/2 + it)| \ll_{\epsilon} q^{\epsilon}\left(1 + |n|^{1/2 + \epsilon} + |t|^{1/2 + \epsilon}\right)\sqrt{y + y^{-1}},\quad \forall 0 < \epsilon < 1/2.
\]
\end{theorem}
Although Theorem~\ref{theorem:huang_xu} is better for weight $n = 0$, we need a version with varying weight for an arithmetic application. In~\cite{MUSICZEHAVI}, the authors studied an arithmetic problem involving the distribution of the roots of a quadratic polynomial to prime moduli, which requires such uniform bounds on the nonspherical Eisenstein series, which is the main motivation behind this work.

All other authors' results cited above (except for Bernstein and Reznikov's) make use of the amplification technique already exhibited in~\cite{IWASAR}. Our result uses a completely different technique, using Sobolev norms, originating in the work of Bernstein and Reznikov~\cite{BERREZ}, who prove an analogue bound to ours for Maass cusp forms in a somewhat greater generality. It should also be mentioned that the nonspherical spectrum arises organically in Bernstein and Reznikov's representation theoretic approach.

In Appendix~\ref{appendix:blomer}, we sketch an alternative proof of Theorem~\ref{theorem:main_theorem} provided by Valentin Blomer, using the approximate functional equation.

\subsection{A high level comparison of the Bernstein-Reznikov paper and ours}
The reader who is acquainted with Bernstein and Reznikov's result and/or with the language of representation theory may ask itself why does Bernstein and Reznikov's result not apply verbatim to the case of the Eisenstein series. This will be discussed from the ground up in the next Chapter, so the reader who is unacquainted with the jargon is advised to skip the following discussion in his or her first reading and jump directly to the next chapter.

Broadly speaking, Bernstein and Reznikov use a Sobolev technique, which takes as input an irreducible representation, an embedding of the representation into the $L^2$-space of the modular variety, and a pair of norms satisfying certain properties, among which, one wants the the norm of right translations coming from a small ball around the identity would be bounded. As their pair of norms, Bernstein and Reznikov choose the standard norm (one which comes from the standard Haar measure), and another norm, coming from the Lie algebra. The first norm has the property of being translation invariant, while the second norm has the property of having the right translation norm (relative to it) very easily bounded. In fact, it reduces to the norm of the adjoint action on the Lie algebra.

Since the Eisenstein series is not square integrable, we consider the truncated Eisenstein series as its $L^2$ embedding. The truncated Eisenstein series is square integrable and has square integrable Lie derivatives, which allows us to construct a pair of norms in an analogue way to that of Bernstein and Reznikov. However, since right translations and truncations don't commute, our standard norm is no longer right shift invariant. It happens to be that a great deal of this paper is dedicated to bounding the norm of right translations coming from a small ball around the identity relative to the standard norm for the truncated Eisenstein series. In particular, this is where we require some arithmetic information regarding the constant term of the Eisenstein series, and is the reason why our work, as opposed to Bernstein and Reznikov's, do not generalize immediately to non-arithmetic quotients of $SL_2(\mathbb{R})$ for which there are no known explicit expressions for the constant terms of Eisenstein series.

On the topic of bounding the norm of right translations coming from a small ball around the identity relative to the standard norm, it is of interest to say that our key insight is that we bound the standard norm of a truncated \textit{translated} Eisenstein series with truncation parameter $T$ using the standard norm of a truncated \textit{non-translated} Eisenstein series with larger truncation parameter, $2T$. Abstractly speaking, one could think about a 1-parameter family of norms over an irreducible Eisenstein representations $V$, which are the standard norms of their truncated embeddings with truncation parameter $T$, indexed by the truncation parameter $T\ge 1$, denoted $||\cdot||_T$. If we denote the right translation operator by $\rho$, ($\rho:SL_2(\mathbb{R})\longrightarrow GL(L^2(\Gamma_0(q)\char`\\SL_2(\mathbb{R})))$), then we managed to prove that, at least for $g\in SL_2(\mathbb{R})$ close enough to the identity, one has for all $v\in V$:
\[
||\rho(g)v||_T \ll ||v||_{2T}.
\]
A question of interest, posed by Reznikov, is whether there exists some function $f:SL_2(\mathbb{R})\longrightarrow \mathbb{R}_{\ge 1}$, with $T(I) = 1$,
having the property that
\[
||\rho(g)v||_T \ll ||v||_{T\cdot f(g)}.
\]
It would probably be even more interesting if one could find such a function that is only dependent on, say, the hyperbolic distance of $g$ from the identity, or any other ``nice" norm on $SL_2(\mathbb{R})$.

\subsection*{Acknowledgements}
This research was supported by the European Research Council (ERC) under the European Union's Horizon 2020 research and innovation programme (Grant agreement No. 786758)

The authors would like to thank Ze\'{e}v Rudnick for his invaluable support, suggestions and very many discussions on this problem, as well as to Andre Reznikov for taking interest in this project and for very illuminating conversations and helpful suggestions. The authors would also like to thank Valentin Blomer for helpful remarks and the alternative proof idea using the approximate functional equation (see Appendix~\ref{appendix:blomer}).

\section{Bernstein and Reznikov's Sobolev technique}
\label{appendix:bernstein_reznikov}
In~\cite{BERREZ}, Bernstein and Reznikov construct Sobolev type bounds on square integrable automorphic forms. Roughly speaking, given a square integrable automorphic form $f$ on $X_0(q)$, and a point $x\in X_0(q)$,
they are able to produce a bound on $|f(x)|$ in terms of the standard norm of $f$ and the norm of its Lie derivatives. While this idea applies seamlessly for cuspidal eigenforms, since these are square integrable functions with square integrable Lie derivatives, it requires some adaptation in order to work for the Eisenstein series.

Before we discuss how we bypass the issue coming from the non-square integrability of the Eisenstein series, we describe how Bernstein and Reznikov's technique applies in the square integrable case.

\subsection{An overview of the Bernstein and Reznikov technique (the square integrable case)}
This Chapter will adhere to the language of representation theory used by Bernstein and Reznikov.

Denote by $X_0(q)$ the space $\Gamma_0(q)\char`\\SL_2(\mathbb{R})$. Let $V$ be a vector space of smooth, square integrable functions over $X_0(q)$. Let $I_x$ be the ``evaluation at $x$" functional, i.e., for $f\in V$, and $x\in X_0(q)$ we define $I_x(f) := f(x)$.

\begin{definition}[The $N$-norm of $I_x$]
Let $N$ be some arbitrary norm on $V$, the $N$-norm of $I_x$ is given by
\[
||I_x||_N = \sup_{v\in V}\dfrac{|I_x(v)|}{||v||_N}.
\]
\end{definition}

\begin{definition}[The P-Norm]
Let $v\in V$, we define the $P$-norm of $v$, denoted $P(v)$, by
\[
P(v) = \int_{X_0(q)}|I_x(v)|^2d\chi(x).
\]
\end{definition}
\textbf{Remark}: when $I_x$ is the pointwise evaluation functional, $P(v)$ is equal to the standard norm on $L^2(X_0(q),\chi)$.
\begin{definition}[The Norm Trace]
Let $P$ be the Hermitian norm defined by the functional $I_x$, and let $Q$ be another Hermitian norm on $V$. Let $(e_n)_{n=1}^{\infty}\subseteq V$ be a basis of $V$ that is orthogonal relative to the norm $Q$. The Norm Trace is
\[
tr(P|Q) = \sum_n \dfrac{P(e_n)}{Q(e_n)}.
\]
\end{definition}
\textbf{Remark}: we do not assume here that the trace is finite, but we will be interested only in pairs of norms for which it will be.

The idea is that the quantity $||I_x||_Q$ may be bounded in terms of the number $tr(P|R)$ as we are about to see. This is done in virtue of The Norm Trace Formula.

\begin{theorem}[The Norm Trace Formula, see~\cite{BERREZ}, Proposition 1.2]
\label{theorem:norm_trace_formula}
Let $(e_n)_{n}\subseteq V$ be a basis of $V$ that is orthogonal relative to the norm $Q$. The Norm Trace Formula (relative to the norms $P$ and $Q$) is the identity
\[
\sum_n \dfrac{P(e_n)}{Q(e_n)} = \int_{X_0(q)}||I_x||_{Q}^2d\chi.
\]
\end{theorem}
\textbf{Remark}: Although the definition of the Norm Trace was contingent on a certain choice of basis for $V$, the Norm Trace Formula shows that it is independent of the choice of the basis.

Now, let $x_0 \in X_0(q)$ be our point of interest. Denote by $S(x_0,d)$ the subset of $X_0(q)$ containing $x_0$, such that for all $y\in S(x_0,d)$ the following inequality holds.
\[
||I_y||_Q^2 \ge \dfrac{||I_{x_0}||_Q^2}{d}.
\]
\textbf{Remark}: Since the $Q$-norm of the pointwise evaluation functional is always lower semi-continuous, for all $d > 1$, the set $S(x_0,d)$ is of positive measure.

In order to bound $||I_{x_0}||_{Q}^2$, we utilize the positivity of the right hand side and lower bound the integral over $X_0(q)$ by an integral over the set $S(x_0,d)$. This is done as follows.

\[
tr(P|Q) = \int_{X_0(q)}||I_x||_{Q}^2 d\chi \ge \int_{S(x_0,d)}||I_x||_{Q}^2 d\chi \ge \int_{S(x_0,d)}\dfrac{||I_{x_0}||_{Q}^2}{d} = ||I_{x_0}||_{Q}^2\dfrac{\chi(S(x_0,d))}{d}.
\]

We obtain Bernstein and Reznikov's pre-Sobolev bound.
\begin{lemma}[pre-Sobolev bound, see~\cite{BERREZ} Theorem 1.1]
\label{lemma:pre_sobolev}
Assume $d > 1$, then
\[
||I_{x}||_Q^2 \le \dfrac{d}{\chi(S(x,d))}tr(P|Q).
\]
\end{lemma}
From the pre-Sobolev bound we derive the inequality, valid for every $v\in V$ and $d > 1$:
\[
|v(x)|^2 \le ||v||_Q^2\dfrac{d}{\chi(S(x,d))}tr(P|Q).
\]
We would like to apply a similar inequality when $v$ is an Eisenstein series of type parameter $t$ and pure even weight $n$. In practice, our bound (the right hand side) would be dependant on our choices of the ambient vector space $V$, the functional $I_x$ and the norm $Q$. In the next section we specialize this abstract bound for specific choices of $V$, $I_x$ and $Q$.

\section{Specializing the pre-Sobolev bound}
\label{section:specialized}
\subsubsection{Fixing $V$}
From this point onward, we assume $q$, the level parameter, is a squarefree integer for reasons which would become clear in Chapter~\ref{section:norm_bound}.

For Bernstein and Reznikov's analysis to work, one requires that the space $V$ is closed under right translations. What this means is that given an automorphic form $f\in V$, and some $g\in SL_2(\mathbb{R})$, the right translate of $f$ by $g$, denoted $\rho(g)(f):X_0(q)\longrightarrow\mathbb{C}$, defined by
\[
\forall h\in SL_2(\mathbb{R}):\quad (\rho(g)f)(h) = f(hg),
\]
is also an element of $V$, i.e. $\rho(g)f\in V$. The smallest possible choices for such vector spaces $V$ are the irreducible representations. These are vector spaces closed under translations that do not have any proper subspaces that are also closed under translations.

It is a fact that the span of the even pure-weight Eisenstein series $E_{\mathfrak{a},2n}(*,1/2 + it)$ forms such an irreducible representation. We denote it by $V^t$, and when there is no confusion simply by $V$, i.e.
\[
V^t = \text{span}_{\mathbb{C}}\{E_{\mathfrak{a},n}(*,1/2 + it)\}_{n\in 2\mathbb{Z}}.
\]
\textbf{Remark}: it is convenient that the even pure-weight Eisenstein series form an irreducible representation, since it means that we can use the same representation $V^t$ to construct point-wise bounds for all pure even-weight Eisenstein series $E_{\mathfrak{a},n}(g,1/2 + it)$ sharing the same type parameter $t$ simultaneously.

\subsubsection{Fixing $I_x$}
Now that we fixed our (irreducible) representation $V$, we would like to construct our functional $I_x:V\longrightarrow\mathbb{C}$. Although we are interested in bounding $|v(x)|$ for $v\in V$ and $x\in X_0(q)$, simply choosing $I_x$ to be the pointwise evaluation functional from the previous Chapter won't work, since then, as in the previous Chpater, the norm $P$ would become the standard norm. This is bad because the Eisenstein series is not square integrable.

The idea is to fix $I_x$ to be the composition of the pointwise evaluation functional with Selberg's truncation operator. The ``truncated Eisenstein series" are functions that are equal to the Eisenstein series up to some height, but are then equal to the Eisenstein series minus its constant term for larger heights, i.e. near the cuspidal zones. We define this rigorously below. The truncated Eisenstein series have the advantage of being square integrable. Before we define the truncation operator we require a few definitions.

We require a definition of height.
\begin{definition}[The height function $h(g)$]
\label{definition:height_function}
Let $g\in SL_2(\mathbb{R})$, we define its height to be
\[
h(g) = \max_{\mathfrak{a}}\max_{\gamma\in\Gamma_0(q)}\{y(\sigma_{\mathfrak{a}}^{-1}\gamma g)\},
\]
where $\sigma_{\mathfrak{a}}$ is the scaling matrix sending the cusp at $\infty$ to $\mathfrak{a}$ (see~\cite[p.40, 2.1]{IWANIECBOOK} for a definition of the scaling matrix), and $y(\sigma_{\mathfrak{a}}^{-1}\gamma g)$ is the Iwasawa $y$-coordinate of $\sigma_{\mathfrak{a}}^{-1}\gamma g$.
\end{definition}

\textbf{Remark}: if $h(g) > 1$, then the maximum over the cusps $\mathfrak{a}$ is obtained for a unique cusp. In this situation we say that $g$ is in the cuspidal zone of that cusp.

Let $g = (x,y,\theta)$ be the Iwasawa coordinates of $g$. Fix $n\in 2\mathbb{Z}$ and $t\in\mathbb{R}$. Then if $\sigma_{\mathfrak{b}}$ is the scaling matrix (taking the cusp at $\infty$ to the cusp at $\mathfrak{b}$), then the Fourier expansion of $E_{\mathfrak{a},n}(\sigma_{\mathfrak{b}}g,1/2 + it)$ in the variable $x$ is given by
\[
E_{\mathfrak{a},n}(\sigma_{\mathfrak{b}}g,1/2 + it) = \sum_ma_{\mathfrak{a},\mathfrak{b},m,n}(y,t)e(mx)e^{in\theta}.
\]

\begin{definition}[The truncation operator]
Let $f \in V$ be an Eisenstein series, and $\mathfrak{a}\in\mathfrak{A}$ be an arbitrary cusp of $X_0(q)$. Denote by
\[
f(\sigma_{\mathfrak{a}}g) = \sum_{m}a_{m,\mathfrak{a}}(y,\theta)e(mx)
\]
the Fourier expansion of $f$ at the cusp $\mathfrak{a}$. The truncation operator with parameter $T\ge 1$, denoted by $\Lambda^T$, where $\Lambda^T:V\longrightarrow L^2(X_0(q),\chi)$, is defined by
\[
\forall f\in V:\quad (\Lambda^Tf)(g) = 
\begin{cases}
f(g)& h(g)\le T,\\
f(g) - a_{0,\mathfrak{b}}(y(\sigma_{\mathfrak{b}}^{-1}\gamma g),\theta(\sigma_{\mathfrak{b}}^{-1}\gamma g))& \gamma\in\Gamma_0(q): y(\sigma_{\mathfrak{b}}^{-1}\gamma g) = h(g) > T.
\end{cases}
\]
\end{definition}
The truncated Eisenstein series, $\Lambda^TE_{\mathfrak{a},n}(g,1/2 + it)$, is equal to the ordinary Eisenstein series, $E_{\mathfrak{a},n}(g,1/2 + it)$, up to height $T$, and to the Eisenstein series minus its constant term for height greater than $T$.

\label{remark:important}
\textbf{Important Remark}: from this point onward, we fix the value of the truncation parameter as follows.
\[
T = 
\begin{cases}
4e^{2\pi}& e^{2\pi}/2\le h(x) \le 2e^{2\pi}\\
e^{2\pi}& \text{otherwise}.
\end{cases}
\]

We are now ready to define the functional $I_x$.
\begin{definition}[The functional $I_x$]
Let $x\in X_0(q)$. We define the function $I_x$, $I_x:V\longrightarrow\mathbb{C}$ by
\[
I_x(v) = (\Lambda^Tv)(x).
\]
\end{definition}

As in the previous Chapter, the norm $P$ is induced from the definition of the functional $I_x$.
\begin{definition}[The norm $P$]
\label{definition:p_norm}
Given an Eisenstein series $f\in V$, we define its $P$ norm by
\[
||f||_P^2 = P(f) := \int_{X_0(q)}|I_x(f)|^2d\chi(x) = \int_{X_0(q)}|\Lambda^Tf(x)|^2d\chi(x).
\]
\end{definition}
We recall that $d\chi(x)$ is ``the standard volume measure" on $X_0(q)$, given by $\dfrac{dxdyd\theta}{y^2}$ in Iwasawa coordinates.

\textbf{Remark}: the norm $P$ is induced by the Hermitian inner product $<\cdot,\cdot>_P:V^2\longrightarrow \mathbb{C}$, defined for a pair of Eisenstein series $\phi,\psi\in V$ by
\[
<\phi,\psi>_P = \int_{X_0(q)}\Lambda^T\phi(x)\overline{\Lambda^T\psi(x)}d\chi(x).
\]

After defining the representation $V$, the functional $I_x$, and the Hermitian norm $P$, we would like to define the norm $Q$. Since our norm $Q$ will be interpolated from an auxiliary norm, $R$, we first define $R$.

\subsubsection{Fixing $R$}
\begin{definition}[The norm $R$]
\label{definition:r_norm}
Denote by $\mathcal{D}$ the differential operator $-2\Delta = \mathfrak{C} - \dfrac{d^2}{d\theta^2}$. Given an Eisenstein series $f\in V$, we define its $R$ norm by
\[
||f||_R^2 = R(f) := \int_{X_0(q)}
\Lambda^T\mathcal{D}f(x)
\overline{\Lambda^Tf(x)}d\chi(x).
\]
\end{definition}
It is easy to verify that $\mathcal{D}:V\longrightarrow V$ is a positive definite operator relative to the inner product $<\cdot,\cdot>_P$, and therefore the norm $R$ is Hermitian as well.

\textbf{Remark}: we denote the pure weight basis, normalized by the constant term, denoted by $(E_{\mathfrak{a},n}(*,1/2+it))_{n\in2\mathbb{Z}}$, by $(e_n)_{n\in 2\mathbb{Z}}$, or $(e_n)_n$ by short.

\textbf{Remark}: it is evident that the norm $R$ has the property that $R(e_n) = (2\lambda + n^2)P(e_n)$ for all even $n$, and that the basis $(e_n)_n$ is orthogonal relative to the norm $R$.

\subsubsection{Fixing $Q$}
Since $\mathcal{D} > 0$, we may define for each real number $s > 0$ the corresponding power of $\mathcal{D}$, i.e. $\mathcal{D}^s$. From this point onward, we fix $s_0 = 1/2 + \epsilon_0$ for some $0 < \epsilon_0 < 1/2$, and denote by $\mathfrak{d}$ the operator $\mathcal{D}^{s_0}$.

We are now ready to define the norm $Q$.
\begin{definition}[The norm $Q$]
\label{definition:q_norm}
Given an Eisenstein series $f\in V$, we define its $Q$ norm by
\[
||f||_Q^2 = Q(f) := \int_{X_0(q)}
\Lambda^T\mathfrak{d}f(x)
\overline{\Lambda^Tf(x)}d\chi(x).
\]
\end{definition}

\textbf{Remark}: the operator $\mathfrak{d}$ is positive by construction. 

\textbf{Remark}: the norm $Q$ preserves the orthogonality of the pure weight basis $(e_n)_n$, and for each $n\in 2\mathbb{Z}$, $Q(e_n) = (2\lambda + n^2)^{1/2 + \epsilon_0}P(e_n)$.

We apply the pre-Sobolev bound with the norms $P$ and $Q$. To simplify our pre-Sobolev bound, we bound $\text{tr}(P|Q)$.

\begin{claim}[Norm Trace bound]
One has
\[
\text{tr}(P|Q) \ll_{\epsilon_0} 1
\]
\end{claim}
\begin{proof}
By definition,
\[
\text{tr}(P|Q) = \sum_{n\in 2\mathbb{Z}}\dfrac{P(e_n)}{Q(e_n)}.
\]
By the previous remark, $Q(e_n) = (2\lambda + n^2)^{1/2 + \epsilon_0}$, and thus
\[
\sum_{n\in 2\mathbb{Z}}\dfrac{P(e_n)}{Q(e_n)} = 
\sum_{n\in 2\mathbb{Z}}\dfrac{1}{(2\lambda + n^2)^{1/2 + \epsilon_0}} \ll_{\epsilon_0} 1.
\]
\end{proof}

Plugging into the pre-Sobolev bound (see Lemma~\ref{lemma:pre_sobolev}), we obtain
\begin{lemma}[Specialized pre-Sobolev]
\label{lemma:pre_sobolev_specialized}
Let $E_{\mathfrak{a},n}(*,1/2 + it)$ be a pure-weight $n$ Eisenstein series, and let $x\in X_0(q)$. Then, with all notations as above, one has
\[
|\Lambda^TE_{\mathfrak{a},n}(x,1/2 + it)|^2 \ll_{\epsilon_0} (2\lambda + n^2)^{1/2 + \epsilon_0}||E_{\mathfrak{a},n}(*,1/2 + it)||_{P}^2\inf_{d > 1}\dfrac{d}{\chi(S(x,d))}.
\]
\end{lemma}

In the next step we discuss structured subsets of $S(x,d)$.

\section{Trading the volume of $S(x,d)$ for structure}
\label{section:volume_for_structure}
We defined the set $S(x,d)$ for some $x\in X_0(q)$ and $d > 1$ to be the set of all $y\in X_0(q)$ satisfying
\[
||I_y||_{Q}^2 \ge \dfrac{||I_{x}||_{Q}^2}{d}.
\]
The set $S(x,d)$ contributes a $1/\chi(S(x,d))$-factor to our pre-Sobolev bound (see Lemma~\ref{lemma:pre_sobolev_specialized}), which means that in order to turn our bound effective, we will require a lower bound on $\chi(S(x,d))$ for some choice of $d$. To obtain such a lower bound, we sacrifice some of the volume of $\chi(S(x,d))$ for structure. This is done as follows.

\begin{definition}[points in the same zone]
\label{definition:same_zone}
Let $x,y\in X_0(q)$. We say that $x$ and $y$ are in the same zone if $h(x),h(y) \le T$, or if $h(x),h(y) > T$, and $x$ and $y$ are in the cuspidal zone of the same cusp.
\end{definition}

Assume $x = yg$ for some $g\in SL_2(\mathbb{R})$ and that $x$ and $y$ are in the same zone. In this case, we have
\[
I_x = I_y\circ \rho(g),
\]
and therefore, for all $v\in V$:
\[
\dfrac{|I_x(v)|^2}{||v||_{Q}^2} = \dfrac{|I_y(\rho(g)v)|^2}{||v||_{Q}^2} = \dfrac{|I_y(\rho(g)v)|^2}{||\rho(g)v||_{Q}^2}\cdot\dfrac{||\rho(g)v||_{Q}^2}{||v||_{Q}^2}\le ||I_y||_{Q}^2||\rho(g)||_{Q}^2.
\]
Taking the supremum on $v$, we obtain the inequality, valid for $x$ and $y$ in the same zone,
\[
||I_x||_Q^2 \le ||I_y||_{Q}^2||\rho(g)||_{Q}^2.
\]
We define the measure $||\rho(g)||_{Q}^2$ for $g\in SL_2(\mathbb{R})$.
\begin{definition}[The $d$-measure]
We define the ``$d$-measure", $d:SL_2(\mathbb{R})\longrightarrow\mathbb{R}_+$ by
\[
\forall g\in SL_2(\mathbb{R}):\quad d(g) = ||\rho(g)||_{Q}^2.
\]
\end{definition}
\textbf{Remark}: the quantity $||\rho(g)||_{Q}$ is defined by $\sup_{0\neq v\in V}\dfrac{||\rho(g)v||_Q}{||v||_Q}$.

We define the ball of radius $d$ centered at $x$ relative to the $d$-measure.
\begin{definition}[The $d$-measure ball]
We define the ball of radius $d$ centered at $x$ relative to the $d$-measure, denoted $B(x,d)$ by
\[
B(x,d) = \{y\in X_0(q)|\exists g\in SL_2(\mathbb{R}) \text{ s.t. } yg = x, d(g) \le d, x\text{ and }y\text{ are in the same zone}\}.
\]
\end{definition}

By definition, for all $y\in B(x,d)$,
\[
||I_y||_Q^2 \ge \dfrac{||I_x||_Q^2}{d},
\]
so that
\[
B(x,d)\subseteq S(x,d),
\]
and therefore
\[
\dfrac{d}{\chi(S(x,d))} \le \dfrac{d}{\chi(B(x,d))}.
\]
We call the ratio $\dfrac{d}{\chi(B(x,d))}$ the continuity ratio. In the next step, we construct an upper bound on $\inf_{d>1}\dfrac{d}{\chi(B(x,d))}$.
\subsection{An upper bound for the continuity ratio}
In this section we prove:
\begin{lemma}
\label{lemma:continuity_constant_full_modular_group}
Let $x\in X_0(q) = \Gamma_0(q)\char`\\ SL_2(\mathbb{R})$, then
\[
\inf_{d>1}\dfrac{d}{\chi(B(x,d))} \ll_{\epsilon} q^{\epsilon}(y(x) + y(x)^{-1}),\quad \forall \epsilon > 0,
\]
where $y(x)$ is the Iwasawa $y$-coordinate of a representative of $x$ taken from the standard fundamental domain for $X_0(q)$.
\end{lemma}
\begin{proof}
Recall that
\[
B(x,d) = \{y\in X_0(q)|\exists g\in SL_2(\mathbb{R}) \text{ s.t. } yg = x, d(g) \le d, x\text{ and }y\text{ are in the same zone}\}.
\]
Our idea is to consider a certain compact set of translates, $\Omega \subseteq SL_2(\mathbb{R})$, for which the right shift norm, i.e. $||\rho(g)||$, will be uniformly bounded for all $g\in \Omega$.

We require a certain definition.
\begin{definition}[The $\delta$-rectangle]
\label{definition:delta_rectangle}
Let $\delta > 0$, we define the $\delta$-rectangle around the identity, denoted $R(\delta)$ to be the set of all elements $g\in SL_2(\mathbb{R})$, such that if the Iwasawa coordinates of $g$ are $x,y,\theta$, (where we choose a representative for $\theta$ in the interval $[0,2\pi]$), then
\[
\max\{|x|,|y-1|,|\theta|\} \le \delta.
\]
\end{definition}

By Claim~\ref{claim:real_q_norm_right_shift_bound}, there exists a certain $\delta > 0$, such that for all $g\in R(\delta)$, one has
\[
||\rho(g)||_Q^2 \ll ||\rho(g)||_P^{1-2\epsilon_0}||\rho(g)||_R^{1+2\epsilon_0},
\]
where $\epsilon_0 > 0$ is as in Definition~\ref{definition:q_norm}. 
By Corollary~\ref{corollary:right_shift_r_norm}, for all $g\in SL_2(\mathbb{R})$ and $\epsilon > 0$, one has
\[
||\rho(g)||_R^2 \ll_{\epsilon} q^{\epsilon}(||Ad(g)||^2 + 1)||\rho(g)||_P^2.
\]
Plugging this back, we find that for all $g\in R(\delta)$, one has
\[
||\rho(g)||_Q^2 \ll_{\epsilon} q^{\epsilon}(||Ad(g)||^{1+2\epsilon_0} + 1)||\rho(g)||_P^2.
\]
Since $||Ad(g)||$ is a continuous function and $\mathfrak{sl}_2(\mathbb{R})$ is finite dimensional, one has for any compact set and, in particular, for all $g\in R(\delta)$: 
\[
||Ad(g)||^2 \ll 1.
\]
Where the implied constant depends on $\delta$, however, since $\delta > 0$ is a fixed constant, the implied constant is absolute.

Plugging this back, we have
\[
||\rho(g)||_Q^2 \ll_{\epsilon} q^{\epsilon}||\rho(g)||_P^2.
\]
Bounding $||\rho(g)||_P^2$ on $R(\delta)$ is harder, and doesn't follow from compactness since $V$ is infinite dimensional as opposed to $\mathfrak{sl}_2(\mathbb{R})$.

Since, in the case of cuspidal representations, $I_x$ is the pointwise evaluation functional, and therefore the norm $P$ is the standard norm, the identity $||\rho(g)||_P^2 = 1$ follows from the basic fact that the standard norm is $SL_2(\mathbb{R})$-right shift invariant, i.e.
\[
\forall g\in SL_2(\mathbb{R}), v\in L^2(X_0(q),\chi):\quad ||\rho(g)v||_P^2 = ||\rho(g)v||^2 = ||v||^2 = ||v||_P^2.
\]
Since the Eisenstein series is not square integrable, we altered the definition of $I_x$, and thereby of the norm $P$. The functional $I_x$ is defined by first truncating and then evaluating at $x$. Therefore, for all $g\in SL_2(\mathbb{R})$ and $v\in V$, by definition
\[
||\rho(g)v||_P^2 = ||\Lambda^T\rho(g)v||^2,\quad ||v||_P^2 = ||\Lambda^Tv||^2.
\]
However, since right translations and truncations don't commute, i.e. $\Lambda^T\rho(g) \neq \rho(g)\Lambda^T$, the same argument as in the square integrable case, using the invariancy of the standard norm relative to right translations, no longer holds. Not all is lost however, as we manage to bound $||\rho(g)||_P^2 \ll 1$ over our $\delta$-rectangles.

Claim~\ref{claim:right_shift_bound} states that for the same $\delta > 0$, we also have
\[
\forall g\in R(\delta):\quad ||\rho(g)||_P^2 \ll 1.
\]
Plugging back, we obtain the following corollary.
\begin{corollary}
\label{corollary:right_shift_q_norm}
With $R(\delta)$ as in Claim~\ref{claim:right_shift_bound}, one has
\[
\forall g\in R(\delta):\quad ||\rho(g)||_{Q}^2 \ll_{\epsilon} q^{\epsilon},\quad \forall \epsilon > 0.
\]
\end{corollary}

\begin{claim}
With all notations as above, there exists a constant $d \ll_{\epsilon} q^{\epsilon}$, and $0 < \delta' \le \delta$, such that 
\[
xR(\delta')\subseteq B(x,d).
\]
\end{claim}
\begin{proof}
The choice of $\delta > 0$ as in Claim~\ref{claim:right_shift_bound} satisfies the additional, following, property. For all $g\in R(\delta)$, if $h(x) > 2T$ then $h(xg) > T \ge 1$ and $x,xg$ are in the same cuspidal zone, and therefore in the same zone (see Definition~\ref{definition:same_zone}). Otherwise if $h(x) \le T/2$, then $h(xg) \le T$ and again $x$ and $xg$ are in the same zone. See Claim~\ref{claim:delta_existence} for reference.

Since, by Remark~\ref{remark:important}, we fix $T$ to be equal to $e^{2\pi}$ by default, or to $4e^{2\pi}$ if $e^{2\pi}/2 < h(x) \le 2e^{2\pi}$, it follows that $x$ and $xg$ are in the same zone for all $g\in R(\delta)$. 

Since the function $g\rightarrow g^{-1}$ is a continuous automorphism of $SL_2(\mathbb{R})$ that fixes the identity, and since $R(\delta)$ contains an open neighborhood of the identity, so does $R(\delta)^{-1}$. Since the set of $\eta$-rectangles $R(\eta)$ with $\eta > 0$ form a filtration of the open neighborhoods of the identity, i.e. for each open neighborhood $U$ of the identity there exists some $\eta_U > 0$ such that $R(\eta_U)\subseteq U$. Fixing $0 < \delta' = \min\{\delta,\eta_{R(\delta)^{-1}}\}$, we find that $R(\delta') \subseteq R(\delta)\cap R(\delta)^{-1}$. Therefore, we also have $R(\delta')^{-1}\subseteq R(\delta)\cap R(\delta)^{-1}$.

Next, we fix $d \ll_{\epsilon} q^{\epsilon}$ to be a number satisfying $||\rho(g)||_Q^2 \le d$ for all $g\in R(\delta)$, which is guaranteed to exist by Corollary~\ref{corollary:right_shift_q_norm}. It follows that for all $g\in R(\delta')\subseteq R(\delta)^{-1}$, one has
\[
||\rho(g^{-1})||_Q^2\le d,
\]
while by the above argument, since $g\in R(\delta')\subseteq R(\delta)$, $x$ and $xg$ are in the same zone, implying that $xg\in B(x,d)$. Since this is true for all $g\in R(\delta')$, we conclude that $xR(\delta')\subseteq B(x,d)$.
\end{proof}

\textbf{Remark}: from this point onward we fix $d \ll_{\epsilon} q^{\epsilon}$ as in the above claim.

\textbf{Remark}: possibly replacing $\delta > 0$ from Claim~\ref{claim:delta_existence} by, the smaller, $\delta' > 0$ from the above claim, we may assume $xR(\delta)\subseteq B(x,d)$.

We are ready to bound the continuity constant and prove Lemma~\ref{lemma:continuity_constant_full_modular_group}. It is instructive to consider first the case of the full modular group, i.e., when $q = 1$.

\begin{claim}
\label{claim:continuity_ratio_full_modular_group}
With all notations as above, for all $x\in X_0(1) = \Gamma_0(1)\char`\\SL_2(\mathbb{R})$, one has
\[
\inf_{d\in\mathbb{R}_+}\dfrac{d}{\chi(B(x,d))} \ll_{\epsilon} q^{\epsilon}y(x),\quad \forall \epsilon > 0.
\]
\end{claim}
\begin{proof}
Fix $\epsilon > 0$, and let $x\in X_0(1)$ be an arbitrary point. Either $y(x) > 2$ or $y(x) \le 2$. Assume first that $y(x) \le 2$. We define the function $z\rightarrow\chi(zR(\delta))$ sending a point $z$ to the volume of the set $zR(\delta)$ in $X_0(1)$. Since the set of points $C:=\{z\in X_0(1)|h(z) \le 2\}$ is compact, and since the function $z\rightarrow\chi(zR(\delta))$ is continuous and positive for every $z$, it follows that it attains a positive minimum on $C$. Since by our assumption on $d$, it is enough to fix $d\ll_{\epsilon} q^{\epsilon}$ so that $xR(\delta) \subseteq B(x,d)$. Therefore, for all such $x$ one has
\[
\inf_{d > 1}\dfrac{d}{\chi(B(x,d))} \ll_{\epsilon} q^{\epsilon}.
\]
Next, we consider the case where $y(x) > 2$. The proof of claim~\ref{claim:delta_existence} shows that for all $\nu > 0$ there exists $\delta(\nu) > 0$ such that $KR(\delta(\nu))\subseteq R(\nu)$. Applying this statement for $\nu = \delta > 0$, we find that there exists $\delta(\delta) = \delta' > 0$ such that $KR(\delta')\subseteq R(\delta)$.

Keeping the same assumptions on $d$ as before, we know that the set $B(x,d)$ contains the rectangle $xR(\delta)$. Let $g_x \in SL_2(\mathbb{R})$ be a representative of $x$ taken from the standard fundamental domain for $X_0(1)$. 
\begin{definition}[The Standard Fundamental Domain for $X_0(1)$]
\label{definition:standard_fundamental_domain}
We define the standard fundamental domain for $X_0(1)$, denoted $\mathcal{F}(\Gamma_0(1)\char`\\ SL_2(\mathbb{R}))\subseteq SL_2(\mathbb{R})$, as follows
\[
\mathcal{F}(\Gamma_0(1)\char`\\ SL_2(\mathbb{R})) = \{g\in SL_2(\mathbb{R}): -1/2\le x(g)\le 1/2, x(g)^2 + y(g)^2 \ge 1\},
\]
where $x(g), y(g)$ denote the Iwasawa $x$ and $y$ coordinates of $g$, respectively.
\end{definition}

If $g_x = n_xa_xk_x$ is the Iwasawa decomposition of $g_x$, we consider the set $k_xR(\delta)$. The above argument shows that $R(\delta') \subseteq k_xR(\delta)$, so that
\[
xR(\delta) \supseteq n_xa_xR(\delta').
\]
Denote by $N_{\delta'},A_{\delta'}$ and $K_{\delta'}$ the sets
\[
N_{\delta'} = 
\left\{
\begin{pmatrix}
1& u\\
0& 1
\end{pmatrix} \in N : |u| \le \delta'
\right\},\quad
A_{\delta'} = 
\left\{
\begin{pmatrix}
y^{1/2}& 0\\
0& y^{-1/2}
\end{pmatrix} \in A : 1 - \delta' \le y \le 1 + \delta'
\right\},
\]
\[
K_{\delta'} = 
\left\{
\begin{pmatrix}
\cos\theta& -\sin\theta\\
\sin\theta& \cos\theta
\end{pmatrix} \in K : |\theta| \le \delta'
\right\}.
\]
Since $N$ is normal in $NA$, we have
\[
ana^{-1} = 
\begin{pmatrix}
y^{1/2}& 0\\
0& y^{-1/2}
\end{pmatrix}
\begin{pmatrix}
1& u\\
0& 1
\end{pmatrix}
\begin{pmatrix}
y^{-1/2}& 0\\
0& y^{1/2}
\end{pmatrix}
=
\begin{pmatrix}
1& yu\\
0& 1
\end{pmatrix},
\]
implying that
\[
a_xN_{\delta'} = N_{y(x)\delta'}a_x.
\]
By definition, $R(\delta') = N_{\delta'}A_{\delta'}K_{\delta'}$, so that
\[
xR(\delta)\supseteq n_xa_xN_{\delta'}A_{\delta'}K_{\delta'} = n_xN_{y(x)\delta'}a_xA_{\delta'}K_{\delta'}.
\]
Therefore, if by abuse of notation we denote by $n_x$ the Iwasawa $x$-coordinate of the point $x$ (as well as the element of $N$), we have
\[
\chi(B(x,d)) \ge \chi(n_xN_{y(x)\delta'}a_xA_{\delta'}K_{\delta'}) = \int_{n_x}^{n_x + \min\{1,y(x)\delta'\}}\int_{y(x)(1-\delta')}^{y(x)(1+\delta')}\int_{-\delta'}^{\delta'}\dfrac{dxdyd\theta}{y^2}
\]
\[
= \min\{1,y(x)\delta'\}y(x)^{-1}\left(\dfrac{1}{1-\delta'} - \dfrac{1}{1 + \delta'}\right)2\delta' \gg y(x)^{-1}.
\]
So that
\[
y(x) > 2\implies \inf_{d>1}\dfrac{d}{\chi(B(x,d))}\ll_{\epsilon} q^{\epsilon}y(x).
\]
Since for all $x\in X_0(1)$, $y(x) \gg 1$, one has
\[
\forall x\in X_0(1):\quad \inf_{d>1}\dfrac{d}{\chi(B(x,d))}\ll_{\epsilon} q^{\epsilon}y(x).
\]
Thus completing the proof of Lemma~\ref{lemma:continuity_constant_full_modular_group} in the case of the full modular group.
\end{proof}

We are now ready to complete the proof of Lemma~\ref{lemma:continuity_constant_full_modular_group}. 

Fix $\epsilon > 0$ and let $x\in\Gamma_0(q)\char`\\ SL_2(\mathbb{R})$. As in the case of the full modular group, we choose a fundamental domain $\mathcal{F}(\Gamma_0(q)\char`\\ SL_2(\mathbb{R}))$, such that~\label{fundamental_domain}
\[
\mathcal{F}(\Gamma_0(q)\char`\\ SL_2(\mathbb{R})) \supseteq 
\mathcal{F},
\]
where $\mathcal{F} = \mathcal{F}(\Gamma_0(1)\char`\\ SL_2(\mathbb{R}))$ is the standard fundamental domain for $X_0(1)$, see Definition~\ref{definition:standard_fundamental_domain}. We also choose $g_x\in \mathcal{F}(\Gamma_0(q)\char`\\ SL_2(\mathbb{R}))$, a representative of $x$. There exists some $\tau \in \Gamma_0(q)\char`\\SL_2(\mathbb{Z})$ with the property that $\tau g_x\in \mathcal{F}$. Claim~\ref{claim:continuity_ratio_full_modular_group} implies that there exists some choice $d = d(\tau g_x) \ll_{\epsilon} q^{\epsilon}$, such that
\[
\dfrac{d}{\text{vol}_{SL_2(\mathbb{R})}(\tau g_xR(\delta)\cap \mathcal{F})} \ll_{\epsilon}  q^{\epsilon}y(\tau g_x) \ll_{\epsilon} q^{\epsilon}(y(x) + y(x)^{-1}),
\]
where $y(x)$ is the Iwasawa $y$-coordinate of $g_x$, the representative of $x$ that is taken from our chosen fundamental domain.

Since the Haar measure on $SL_2(\mathbb{R})$ is left shift invariant, $\text{vol}_{SL_2(\mathbb{R})}(\tau g_xR(\delta)\cap \mathcal{F}) = \text{vol}_{SL_2(\mathbb{R})}(g_xR(\delta)\cap \tau^{-1}\mathcal{F})$. Since these sets project one to one onto $X_0(q)$, we also have
\[
\chi(xR(\delta)\cap \tau^{-1}\mathcal{F}) = \text{vol}_{SL_2(\mathbb{R})}(g_xR(\delta)\cap \tau^{-1}\mathcal{F}) = \text{vol}_{SL_2(\mathbb{R})}(\tau g_xR(\delta)\cap \mathcal{F}).
\]
Therefore,
\[
\dfrac{d}{\chi(xR(\delta)\cap \tau^{-1}\mathcal{F})}\ll_{\epsilon} q^{\epsilon}(y(x) + y(x)^{-1}).
\]
Since $B(x,d)\supseteq xR(\delta)\supseteq xR(\delta)\cap \tau^{-1}\mathcal{F}$, we trivially have
\[
\chi(B(x,d)) \ge \chi(xR(\delta)\cap \tau^{-1}\mathcal{F}).
\]
Therefore,
\[
\dfrac{d}{\chi(B(x,d))}\ll_{\epsilon} q^{\epsilon}(y(x) + y(x)^{-1}),
\]
which completes the proof.
\end{proof}

\section{The Sobolev bound}
The following is a corollary of Lemma~\ref{lemma:pre_sobolev_specialized} and
Lemma~\ref{lemma:continuity_constant_full_modular_group}.
\begin{corollary}
\label{corollary:sobolev_unplugged}
For all cusps $\mathfrak{a}\in\mathfrak{A}$, $n\in 2\mathbb{Z}$, $t\in\mathbb{R}$, $x\in X_0(q)$, $1/2 > \epsilon_0 > 0$ and $\epsilon > 0$, one has:
\[
|\Lambda^TE_{\mathfrak{a},n}(x,1/2 + it)|^2 \ll_{\epsilon} 
q^{\epsilon}(\lambda + n^2)^{1/2 + \epsilon_0}||E_{\mathfrak{a},n}(*,1/2 + it)||_{P}^2\left(y(x) + y(x)^{-1}\right).
\]
\end{corollary}

By Claim~\ref{claim:truncated_eisenstein_norm_bound}, one has
\[
||E_{\mathfrak{a},n}(*,1/2 + it)||_P^2 \ll_{\epsilon} q^{\epsilon}\left(1 + \log(1 + |n/2|) + \log(1 + |t|)\right).
\]
Plugging into the above Corollary, we obtain:
\begin{corollary}
With the same notations as in Corollary~\ref{corollary:sobolev_unplugged}, one has
\[
|\Lambda^TE_{\mathfrak{a},n}(x,1/2 + it)|^2 \ll_{\epsilon,\epsilon_0} 
q^{\epsilon}\left(\lambda^{1/2 + \epsilon_0} + |n|^{1 + \epsilon_0}\right)\left(y(x) + y(x)^{-1}\right).
\]
\end{corollary}

\begin{theorem}[Sobolev bound for pure-weight $n$ Eisenstein series]
\label{theorem:sobolev_bound}
Let $E_{\mathfrak{a},n}(*,1/2 + it)$ be a pure-even-weight $n$ Eisenstein series of the cusp $\mathfrak{a}\in\mathfrak{A}$. Let $x\in X_0(q)$, and let $1/2 > \epsilon > 0$. Then
\[
|E_{\mathfrak{a},n}(x,1/2 + it)|^2 \ll_{\epsilon} q^{\epsilon}\left(\lambda^{1/2 + \epsilon} + |n|^{1 + \epsilon}\right)\left(y(x) + y(x)^{-1}\right).
\]
\end{theorem}
\begin{proof}
By the previous corollaries, assume $E_{\mathfrak{a},n}(*,1/2 + it)$ is a pure-even-weight $n$ Eisenstein series of the cusp $\mathfrak{a}\in\mathfrak{A}$. Let $x\in X_0(q)$, and let $1/2 > \epsilon > 0$. Then
\[
|\Lambda^TE_{\mathfrak{a},n}(x,1/2 + it)|^2 \ll_{\epsilon} q^{\epsilon}\left(\lambda^{1/2 + \epsilon} + |n|^{1 + \epsilon}\right)\left(y(x) + y(x)^{-1}\right).
\]
By the definition of the truncation operator, if $x\in X_0(q)$ satisfies $h(x) \le T$, then $\Lambda^TE_{\mathfrak{a},n}(x,1/2 + it) = E_{\mathfrak{a},n}(x,1/2 + it)$, and the theorem follows.

Otherwise, assume $h(x) > T$. Then there exists some cusp $\mathfrak{b}\in\mathfrak{A}$ such that the point $x$ is in the cuspidal zone of the cusp $\mathfrak{b}$. By the definition of the truncation operator, if $g_x\in SL_2(\mathbb{R})$ is a representative of $x\in X_0(q)$ in some fundamental domain, then
\[
|E_{\mathfrak{a},n}(x,1/2 + it) - \Lambda^TE_{\mathfrak{a},n}(x,1/2 + it)| = |c_{\mathfrak{a},\mathfrak{b},n}(\sigma_{\mathfrak{b}}^{-1}\gamma g_x,1/2 + it)|,
\]
where $c_{\mathfrak{a},\mathfrak{b},n}(*,1/2 + it)$ is the constant term of the Fourier series of $E_{\mathfrak{a},n}(*,1/2 + it)$, when expanded around the cusp $\mathfrak{b}$, and $\gamma\in\Gamma_0(q)$ is chosen such that $y(\sigma_{\mathfrak{b}}^{-1}\gamma g_x) = h(g_x) = h(x)$. By definition (see Section~\ref{section:norm_bound}), one has
\[
c_{\mathfrak{a},\mathfrak{b},n}(x,1/2 + it) = e^{in\theta(x)}\left(\delta_{\mathfrak{a} = \mathfrak{b}}y(x)^{1/2 + it} + \phi_{\mathfrak{a},\mathfrak{b}}(1/2 + it)\alpha(n,1/2 + it)y(x)^{1/2 - it}\right).
\]
Since $|\phi_{\mathfrak{a},\mathfrak{b}}(1/2 + it)| \le 1$, and $|\alpha(n,1/2 + it)| = 1$ (see Claim~\ref{claim:truncated_eisenstein_norm_bound} and Proposition~\ref{proposition:scattering_matrix} for reference), we find that
\[
|c_{\mathfrak{a},\mathfrak{b},n}(\sigma_{\mathfrak{b}}^{-1}\gamma g_x,1/2 + it)| \ll h(x)^{1/2}.
\]
By the triangle and mean inequality,
\[
|E_{\mathfrak{a},n}(x,1/2 + it)|^2 \ll |\Lambda^TE_{\mathfrak{a},n}(x,1/2 + it)|^2 + |E_{\mathfrak{a},n}(x,1/2 + it) - \Lambda^TE_{\mathfrak{a},n}(x,1/2 + it)|^2,
\]
plugging in our bounds for $|\Lambda^TE_{\mathfrak{a},n}(x,1/2 + it)|^2$ and $|E_{\mathfrak{a},n}(x,1/2 + it) - \Lambda^TE_{\mathfrak{a},n}(x,1/2 + it)|^2$, we obtain
\[
|E_{\mathfrak{a},n}(x,1/2 + it)|^2 \ll_{\epsilon} q^{\epsilon}\left(\lambda^{1/2 + \epsilon} + |n|^{1 + \epsilon}\right)\left(y(x) + y(x)^{-1}\right) + h(x).
\]
Finally, we note that $h(x) = \max_{\mathfrak{b}}\max_{\gamma\in\Gamma_0(q)}\{y(\sigma_{\mathfrak{b}}^{-1}\gamma x)\}$, while
\[
\sigma_{\mathfrak{b}}^{-1}\gamma = 
\begin{pmatrix}
\sqrt{w_{\mathfrak{b}}}^{-1}& 0\\
0& \sqrt{w_{\mathfrak{b}}}
\end{pmatrix}\tau,
\]
where $w_{\mathfrak{b}}$ is the width of the cusp $\mathfrak{b}$ and $\tau\in SL_2(\mathbb{Z})$. Since for all $\tau\in SL_2(\mathbb{Z})$, $y(\tau x) \ll y(x) + y(x)^{-1}$, and since for all $z\in \mathbb{H}$,
\[
y\left(\begin{pmatrix}
\sqrt{w_{\mathfrak{b}}}^{-1}& 0\\
0& \sqrt{w_{\mathfrak{b}}}
\end{pmatrix}z
\right) = y(z)/w_{\mathfrak{b}},
\]
and since for all cusps $\mathfrak{b}\in\mathfrak{A}$, $w_{\mathfrak{b}} \ge 1$, we find that $h(x) \ll y(x) + y(x)^{-1}$. 

Plugging this back, our proof is complete.
\end{proof}
This proves our Main Theorem, see Theorem~\ref{theorem:main_theorem}.

\section{Bounding the norm of the truncated Eisenstein series}
\label{section:norm_bound}
In this section we prove Claim~\ref{claim:truncated_eisenstein_norm_bound}, used in the proof of Theorem~\ref{theorem:sobolev_bound}.
\begin{claim}
\label{claim:truncated_eisenstein_norm_bound}
One has for all $n\in 2\mathbb{Z}$ and all $\epsilon > 0$
\[
||E_{\mathfrak{a},n}(*,1/2 + it)||_P^2 \ll_{\epsilon} q^{\epsilon}\left(1 + \log(1 + |n/2|) + \log(1 + |t|)\right).
\]
\end{claim}

\begin{proof}
A critical identity used in the proof is the Maass-Selberg relations, which involve the scattering matrix.

The scattering matrix is the matrix of constant terms of the spherical Eisenstein series, that is the matrix $\Phi(s) = (\phi_{\mathfrak{a},\mathfrak{b}}(s))$, indexed by pairs of cusps of $X_0(q)$, such that for each pair of cusps $\mathfrak{a},\mathfrak{b}$, the Fourier expansion of the Eisenstein series $E_{\mathfrak{a},0}(*,1/2 + it)$ at the cusp $\mathfrak{b}$ has the form
\[
E_{\mathfrak{a},0}(\sigma_{\mathfrak{b}}g,1/2 + it) = \delta_{\mathfrak{a} = \mathfrak{b}}y^{1/2 + it} + \phi_{\mathfrak{a},\mathfrak{b}}(1/2 + it)y^{1/2 - it} + \text{ higher order terms}.
\]
It is easily deduced from the above expansion, using the raising and lowering operators, that the Fourier expansion of the weight $n$ Eisenstein series of the cusp $\mathfrak{a}$ at the cusp $\mathfrak{b}$ has the form
\[
E_{\mathfrak{a},n}(\sigma_{\mathfrak{b}}g,1/2 + it) = e^{in\theta}\left(\delta_{\mathfrak{a} = \mathfrak{b}}y^{1/2 + it} + \phi_{\mathfrak{a},\mathfrak{b}}(1/2 + it)\alpha(n,1/2 + it)y^{1/2 - it} + \text{ higher order terms}\right),
\]
where 
\[
\alpha(n,1/2 + it) = \dfrac{\Gamma(1/2 - it + |n/2|)}{\Gamma(1/2 - it)}\cdot\dfrac{\Gamma(1/2 + it)}{\Gamma(1/2 + it + |n/2|)}.
\]
\textbf{Remark}: clearly, $|\alpha(n,1/2 + it)| = 1$. 

In order to deduce Claim~\ref{claim:truncated_eisenstein_norm_bound} we require explicit expressions for the $\phi_{\mathfrak{a},\mathfrak{b}}(s)$. For squarefree $q$, these are given by
\begin{proposition}[see e.g. Cakoni, Chanillo and Fioralba~\cite{CCF}]
\label{proposition:scattering_matrix}
Let $q$ be a squarefree positive integer and define the matrix $N_p(s)$ by
\[
N_p(s) := (p^{2s} - 1)^{-1}
\begin{pmatrix}
p - 1& p^s - p^{1-s}\\
p^s - p^{1-s}& p - 1
\end{pmatrix}.
\]
Then, the scattering matrix, $\Phi(s)$, is given by
\[
\Phi(s) = \psi(s)\bigotimes_{p|q}N_p(s),
\]
where $\psi(s)$ is given by 
\[
\psi(1/2 + it) = \pi^{1/2}\dfrac{\Gamma(it)}{\Gamma(1/2 + it)}\dfrac{\zeta(2it)}{\zeta(1 + 2it)},
\]
and $\bigotimes$ denotes the matrix tensor product.
\end{proposition}
\textbf{Remark}: by the functional equation for $\zeta(s)$, the function $\psi(1/2 + it)$ is smooth for $t\in\mathbb{R}$, and $|\psi(1/2 + it)| = 1$.

\textbf{Remark}: the matrix $\Phi(s)$ is unitary on the critical line and equals to the identity at the central point (see Kubota~\cite[Page 43, Theorem 4.3.4]{KUBOTA}).

Adhering to the above notations, we may sometimes identify cusps as divisors of $q$, e.g. we may write $\mathfrak{a} = p_1\cdot...\cdot p_k$. 

For all $\mathfrak{a} \in\mathfrak{A}$ one has
\[
\phi_{\mathfrak{a},\mathfrak{a}}(s) = \psi(s)\prod_{p|q}\dfrac{p-1}{p^{2s} - 1}.
\]

We are ready to state the Maass-Selberg relations.
\begin{theorem}[Maass-Selberg relations]
\label{theorem:maass_selberg}
With all notations as above, one has
\[
||\Lambda^TE_{\mathfrak{a},n}(*,1/2 + it)||^2 =
2\log T - \sum_{\mathfrak{b}\in \mathfrak{A}}\dfrac{d}{dt}\left(\alpha(n,1/2 + it)\phi_{\mathfrak{a},\mathfrak{b}}(1/2 + it)\right)\overline{\alpha(n,1/2 + it)\phi_{\mathfrak{a},\mathfrak{b}}(1/2 + it)}
\]
\[
+ \dfrac{\overline{\alpha(n,1/2 + it)\phi_{\mathfrak{a},\mathfrak{a}}(1/2 + it)}T^{2it} - \alpha(n,1/2 + it)\phi_{\mathfrak{a},\mathfrak{a}}(1/2 + it)T^{-2it}}{2it}.
\]
\end{theorem}
\textbf{Remark}: for $t=0$ the corresponding Maass-Selberg relation is obtained by taking the limit.

The proof of Claim~\ref{claim:truncated_eisenstein_norm_bound} involves a very straight forward process of bounding the second and third terms of the right hand side of the Maass-Selberg relations.

We begin with the third term.
\begin{claim}
\label{claim:third_term}
\[
\left|\dfrac{\overline{\alpha(n,1/2 + it)\phi_{\mathfrak{a},\mathfrak{a}}(1/2 + it)}T^{2it} - \alpha(n,1/2 + it)\phi_{\mathfrak{a},\mathfrak{a}}(1/2 + it)T^{-2it}}{2it}\right| \ll \log^2(1 + q) + \log(1 + |n/2|).
\]
\end{claim}
\begin{proof}
By the definition of $\psi$ and $\alpha$, we have
\[
\overline{\alpha(n,1/2 + it)} = \alpha(n,1/2 - it),\quad \overline{\phi_{\mathfrak{a},\mathfrak{a}}} = \phi_{\mathfrak{a},\mathfrak{a}}(1/2 - it).
\]
Therefore, our objective reduces to bounding
\[
\left|\dfrac{\alpha(n,1/2 - it)\phi_{\mathfrak{a},\mathfrak{a}}(1/2 - it)T^{2it} - \alpha(n,1/2 + it)\phi_{\mathfrak{a},\mathfrak{a}}(1/2 + it)T^{-2it}}{2it}\right|.
\]
We start off by assuming that $|t| > 1$. By the remark after Theorem~\ref{theorem:maass_selberg}, $|\alpha(n,1/2 - it)\phi_{\mathfrak{a},\mathfrak{a}}(1/2 - it)T^{2it}| = 1$, so that
\[
\left|\dfrac{\overline{\alpha(n,1/2 + it)\phi_{\mathfrak{a},\mathfrak{a}}(1/2 + it)}T^{2it} - \alpha(n,1/2 + it)\phi_{\mathfrak{a},\mathfrak{a}}(1/2 + it)T^{-2it}}{2it}\right| \ll 1.
\]
Otherwise, assume $|t|\le 1$. By Lagrange's theorem, there exists some $|r| \le |t|$ such that
\[
\left|\dfrac{\alpha(n,1/2 - it)\phi_{\mathfrak{a},\mathfrak{a}}(1/2 - it)T^{2it} - \alpha(n,1/2 + it)\phi_{\mathfrak{a},\mathfrak{a}}(1/2 + it)T^{-2it}}{2it}\right| = \left|\dfrac{d}{dr}\alpha(n,1/2 + ir)\phi_{\mathfrak{a},\mathfrak{a}}(1/2 + ir)T^{-2ir}\right|.
\]
We may naively bound this expression by the absolute value of the supremum of the derivative of $\alpha(n,1/2 + ir)\phi_{\mathfrak{a},\mathfrak{a}}(1/2 + ir)T^{-2ir}$ by $r$ where $|r| < 1$. Clearly, we have
\[
\dfrac{d}{dr}\alpha(n,1/2 + ir)\phi_{\mathfrak{a},\mathfrak{a}}(1/2 + ir)T^{-2ir} = 
(\dfrac{d}{dr}\alpha(n,1/2 + ir))\phi_{\mathfrak{a},\mathfrak{a}}(1/2 + ir)T^{-2ir} + 
\]
\[
\alpha(n,1/2 + ir)(\dfrac{d}{dr}\phi_{\mathfrak{a},\mathfrak{a}}(1/2 + ir))T^{-2ir} - 2i\log T\alpha(n,1/2 + ir)\phi_{\mathfrak{a},\mathfrak{a}}(1/2 + ir)T^{-2ir} = (1) + (2) + (3).
\]
We begin by upper bounding $|(3)|$. Since $|\alpha(n,1/2 + ir)\phi_{\mathfrak{a},\mathfrak{a}}(1/2 + ir)T^{-2ir}| \le 1$, we have
\[
|(3)| = |2i\log T\alpha(n,1/2 + ir)\psi(1/2 + ir)T^{-2ir}| \ll \log T \ll 1.
\]
As for $|(2)|$, $|\alpha(n,1/2 + ir)| = |T^{-2ir}| = 1$, so that $|(2)| \le |\dfrac{d}{dr}\phi_{\mathfrak{a},\mathfrak{a}}(1/2 + ir)|$. By the definition of $\phi_{\mathfrak{a},\mathfrak{a}}(1/2 + ir)$, we have
\[
\phi_{\mathfrak{a},\mathfrak{a}}(1/2 + ir) = \psi(1/2 + ir)\prod_{p|q}\dfrac{p-1}{p^{1 + 2ir} - 1},
\]
and therefore
\[
\dfrac{d}{dr}\phi_{\mathfrak{a},\mathfrak{a}}(1/2 + ir) = (\dfrac{d}{dr}\psi(1/2 + ir))\prod_{p|q}\dfrac{p-1}{p^{1 + 2ir} - 1} \]
\[
+ \psi(1/2 + ir)\sum_{p_1|q}(\dfrac{d}{dr}(p_1^{1 + 2ir} - 1)^{-1})\prod_{p|q/p_1}(p^{1 + 2ir} - 1)^{-1}
\prod_{p|q}(p - 1) = (2.1) + (2.2).
\]
As for $(2.1)$, we notice that
\[
(\dfrac{d}{dr}\psi(1/2 + ir))\prod_{p|q}\dfrac{p-1}{p^{1 + 2ir} - 1} = \dfrac{\frac{d}{dr}\psi(1/2 + ir)}{\psi(1/2 + ir)}\phi_{\mathfrak{a},\mathfrak{a}}(1/2 + ir),
\]
and therefore, since $|\phi_{\mathfrak{a},\mathfrak{a}}(1/2 + ir)| \le 1$, we obtain the bound
\[
|(2.1)| \le \left|\dfrac{\frac{d}{dr}\psi(1/2 + ir)}{\psi(1/2 + ir)}\right|,
\]
which is the absolute value of the logarithmic derivative of the $\psi$ function. By the triangle inequality and the definition of $\psi$, we have
\[
\left|\dfrac{\psi'(1/2 + ir)}{\psi(1/2 + ir)}\right| \le \left|\dfrac{\Gamma(ir)'}{\Gamma(ir)} - \dfrac{\Gamma(1/2 + ir)'}{\Gamma(1/2 + ir)} + 2\dfrac{\zeta(2ir)'}{\zeta(2ir)} - 2\dfrac{\zeta(1 + 2ir)'}{\zeta(1 + 2ir)}\right|.
\]
Since $\Gamma(1/2 + ir)$ and $\zeta(2ir)$ are non-vanishing analytic functions for $|r|\le 1$, their logarithmic derivatives are analytic as well, and are, therefore, bounded by a constant on the compact set $|r| \le 1$. As for $\Gamma(ir)$ and $\zeta(1 + 2ir)$, the two functions have a simple pole at $r=0$, and are otherwise smooth and non-vanishing for $0 < |r| \le 1$. For the Gamma function this is true since it is smooth and nowhere vanishing. Thus, their ratio can be extended into a non-vanishing analytic function on $|r|\le 1$, and its logarithmic derivative is therefore bounded on $|r|\le 1$ as well. We conclude that 
\[
|\dfrac{\Gamma(ir)'}{\Gamma(ir)} - \dfrac{\Gamma(1/2 + ir)'}{\Gamma(1/2 + ir)} + 2\dfrac{\zeta(2ir)'}{\zeta(2ir)} - 2\dfrac{\zeta(1 + 2ir)'}{\zeta(1 + 2ir)}| \ll 1
\]
for $|r| < 1$, and therefore that $|(2.1)| \ll 1$. 

We proceed to analyze $(2.2)$. Since
\[
\dfrac{d}{dr}(p^{1+2ir} - 1)^{-1} = -2i\log p(p^{1+2ir} - 1)^{-2}p^{1+2ir},
\]
we may write $(2.2)$ as
\[
-2i\psi(1/2 + ir)\prod_{p|q}\dfrac{p-1}{p^{1 + 2ir} - 1}\sum_{p|q}\dfrac{p^{1+2ir}\log p}{p^{1+2ir} - 1} = -2i\phi_{\mathfrak{a},\mathfrak{a}}(1/2 + ir)\sum_{p|q}\dfrac{p^{1+2ir}\log p}{p^{1+2ir} - 1}.
\]
Therefore, since $|\phi_{\mathfrak{a},\mathfrak{a}}(1/2 + ir)| \le 1$, we may bound
\[
|(2.2)| \ll |\sum_{p|q}\dfrac{p^{1+2ir}\log p}{p^{1+2ir} - 1}| \le \sum_{p|q}\dfrac{p\log p}{p - 1} \ll \log^2(1 + q),
\]
where we bound $\log p$ by $\log(1 + q)$ instead of $\log q$ in order to accommodate the case where $q = 1$. Combining our bounds on $|(2.1)|$ and $|(2.2)|$ we obtain
\[
|(2)| \ll \log^2(1 + q).
\]
To handle $(1)$, we first observe that
\[
|(1)| \ll |\dfrac{d}{dr}\alpha(n,1/2 + ir)|,
\]
and thus we are left with bounding $|\dfrac{d}{dr}\alpha(n,1/2 + ir)|$. Since we shall need a bound on $|\dfrac{d}{dr}\alpha(n,1/2 + ir)|$ with $r$ arbitrary in the next section, we prove a more general claim.
\begin{claim}
\label{claim:alpha_bound}
For all $n\in\mathbb{Z}$ and $r\in\mathbb{R}$, one has
\[
|\dfrac{d}{dr}\alpha(n,1/2 + ir)| \ll 1 + \log(1 + |n/2|).
\]
\end{claim}
\begin{proof}
We recall the definition of $\alpha(n,s)$.
\[
\alpha(n,s) = \dfrac{\Gamma(s)}{\Gamma(1 - s)}\cdot\dfrac{\Gamma(1 - s + |n/2|)}{\Gamma(s + |n/2|)}.
\]
So that using the recurrence $\Gamma(s + 1) = s\Gamma(s)$, we have
\[
\alpha(n,s) = \prod_{k=0}^{|n|/2 - 1}\dfrac{1 - s + k}{s + k}.
\]
Therefore,
\[
\frac{d}{ds}\alpha(n,s) = \dfrac{\frac{d}{ds}\alpha(n,s)}{\alpha(n,s)}\cdot\alpha(n,s) = -\alpha(n,s)\cdot\sum_{k=0}^{|n|/2-1}\left(\dfrac{1 + 2k}{(1 - s + k)(s + k)}\right).
\]
For $r\in \mathbb{R}$, $|\alpha(n,s)| = 1$, and therefore,
\[
\left|\frac{d}{dr}\alpha(n,1/2 + ir)\right| \le \sum_{k=0}^{|n|/2-1}\left(\dfrac{1 + 2k}{(1/2 - ir + k)(1/2 + ir + k)}\right) = \sum_{k=0}^{|n|/2-1}\left(\dfrac{1 + 2k}{r^2 + (1/2 + k)^2}\right)
\]
\[
\le 2\sum_{k = 0}^{|n/2| - 1}\dfrac{1}{1/2 + k} \ll 1 + \log(1 + |n|/2),
\]
which is what we wanted to prove.
\end{proof}
Applying the previous claim, we find that $|(1)| \ll 1 + \log(1 + |n/2|)$.

Combining our bounds for $|(1)|, |(2)|$ and $|(3)|$, we conclude that for $|r| < 1$,
\[
|\dfrac{d}{dr}\alpha(n,1/2 + ir)\psi(1/2 + ir)T^{-2ir}| \ll \log^2(1 + q) + \log(1 + |n/2|),
\]
which completes the proof of Claim~\ref{claim:third_term}.
\end{proof}
Next, we would like to construct a bound for the second term in the Maass-Selberg formula, see Theorem~\ref{theorem:maass_selberg}.
\begin{claim}
\label{claim:second_term}
One has for all $t\in\mathbb{R}$ and $\epsilon > 0$
\[
\left|\sum_{\mathfrak{b}\in \mathfrak{A}}\dfrac{d}{dt}\left(\alpha(n,1/2 + it)\phi_{\mathfrak{a},\mathfrak{b}}(1/2 + it)\right)\overline{\alpha(n,1/2 + it)\phi_{\mathfrak{a},\mathfrak{b}}(1/2 + it)}\right| \ll_{\epsilon} q^{\epsilon}\left(1 + \log(1 + |t|) + \log(1 + |n/2|)\right).
\]
\end{claim}
\begin{proof}
Applying the triangle inequality, we seek to bound
\[
\sum_{\mathfrak{b}\in \mathfrak{A}}\left|\dfrac{d}{dt}\left(\alpha(n,1/2 + it)\phi_{\mathfrak{a},\mathfrak{b}}(1/2 + it)\right)\overline{\alpha(n,1/2 + it)\phi_{\mathfrak{a},\mathfrak{b}}(1/2 + it)}\right|.
\]
Since $|\phi_{\mathfrak{a},\mathfrak{b}}(1/2 + it)| \le |\alpha(n,1/2 + it)| = 1$, we bound by
\[
\sum_{\mathfrak{b}\in \mathfrak{A}}\left|\dfrac{d}{dt}\left(\alpha(n,1/2 + it)\phi_{\mathfrak{a},\mathfrak{b}}(1/2 + it)\right)\right|.
\]
And since
\[
\dfrac{d}{dt}\left(\alpha(n,1/2 + it)\phi_{\mathfrak{a},\mathfrak{b}}(1/2 + it)\right) = (\dfrac{d}{dt}\alpha(n,1/2 + it))\phi_{\mathfrak{a},\mathfrak{b}}(1/2 + it) + \alpha(n,1/2 + it)(\dfrac{d}{dt}\phi_{\mathfrak{a},\mathfrak{b}}(1/2 + it)),
\]
we may bound
\[
|\dfrac{d}{dt}\left(\alpha(n,1/2 + it)\phi_{\mathfrak{a},\mathfrak{b}}(1/2 + it)\right)| \le |\dfrac{d}{dt}\alpha(n,1/2 + it)| + |\dfrac{d}{dt}\phi_{\mathfrak{a},\mathfrak{b}}(1/2 + it)|.
\]
By Claim~\ref{claim:alpha_bound}, we have
\[
|\dfrac{d}{dt}\alpha(n,1/2 + it)| \ll 1 + \log(1 + |n/2|),
\]
as for $|\dfrac{d}{dt}\phi_{\mathfrak{a},\mathfrak{b}}(1/2 + it)|$, we first write out $\phi_{\mathfrak{a},\mathfrak{b}}(s)$
\[
\phi_{\mathfrak{a},\mathfrak{b}}(s) = \psi(s)\prod_{p|q}(p^{2s} - 1)^{-1}\prod_{\substack{p|\mathfrak{a}\wedge p|\mathfrak{b}\\ \vee \\ p\nmid\mathfrak{a}\wedge p\nmid\mathfrak{b}}}(p-1)
\prod_{\substack{p\mid \mathfrak{a}\wedge p\nmid\mathfrak{b}\\\vee \\
p\mid\mathfrak{b}\wedge p\nmid\mathfrak{a}}}(p^s - p^{1-s}),
\]
and therefore
\[
\phi_{\mathfrak{a},\mathfrak{b}}'(s) = \psi'(s)\prod_{p|q}(p^{2s} - 1)^{-1}\prod_{\substack{p|\mathfrak{a}\wedge p|\mathfrak{b}\\ \vee \\ p\nmid\mathfrak{a}\wedge p\nmid\mathfrak{b}}}(p-1)
\prod_{\substack{p\mid \mathfrak{a}\wedge p\nmid\mathfrak{b}\\\vee \\
p\mid\mathfrak{b}\wedge p\nmid\mathfrak{a}}}(p^s - p^{1-s})
\]
\[
+ \psi(s)\sum_{p_1|q}\dfrac{d}{ds}(p_1^{2s} - 1)^{-1}\prod_{p|(q/p_1)}(p^{2s} - 1)^{-1}\prod_{\substack{p|\mathfrak{a}\wedge p|\mathfrak{b}\\ \vee \\ p\nmid\mathfrak{a}\wedge p\nmid\mathfrak{b}}}(p-1)
\prod_{\substack{p\mid \mathfrak{a}\wedge p\nmid\mathfrak{b}\\\vee \\
p\mid\mathfrak{b}\wedge p\nmid\mathfrak{a}}}(p^s - p^{1-s})
\]
\[
+ \psi(s)\prod_{p|q}(p^{2s} - 1)^{-1}\prod_{\substack{p|\mathfrak{a}\wedge p|\mathfrak{b}\\ \vee \\ p\nmid\mathfrak{a}\wedge p\nmid\mathfrak{b}}}(p-1)
\sum_{\substack{p_1\mid \mathfrak{a}\wedge p_1\nmid\mathfrak{b}\\\vee \\
p_1\mid\mathfrak{b}\wedge p_1\nmid\mathfrak{a}}}\dfrac{d}{ds}(p_1^s - p_1^{1-s})
\prod_{\substack{p\mid \mathfrak{a}\wedge p\nmid\mathfrak{b}\\\vee \\
p\mid\mathfrak{b}\wedge p\nmid\mathfrak{a}\\ p\neq p_1}}(p^s - p^{1-s})
\]
\[
= (1) + (2) + (3).
\]
We begin by estimating $(1)$. Note that
\[
(1) = \dfrac{\psi'(s)}{\psi(s)}\phi_{\mathfrak{a},\mathfrak{b}}(s),
\]
and since $\Phi(s)$ is a unitary matrix on the critical line, $|\phi_{\mathfrak{a},\mathfrak{b}}(s)| \le 1$, which allows us to bound $|(1)|$ by the logarithmic derivative of $\psi(s)$. In Claim~\ref{claim:third_term} we analyzed the logarithmic derivative of $\psi(s)$ for $s = 1/2 + it$ with $|t| \le 1$. The same calculation is valid for all $t\in\mathbb{R}$ and yields,
\[
\left|\dfrac{\psi'(1/2 + it)}{\psi(1/2 + it)}\right| = 
\left|\dfrac{\Gamma(it)'}{\Gamma(it)} - \dfrac{\Gamma(1/2 + it)'}{\Gamma(1/2 + it)} + 2\dfrac{\zeta(2it)'}{\zeta(2it)} - 2\dfrac{\zeta(1 + 2it)'}{\zeta(1 + 2it)}\right|.
\]
The same reasoning as in Claim~\ref{claim:third_term} implies that this expression is upper bounded by a constant for $|t| \le 1$. As for $|t| > 1$, we appeal to the logarithmic derivative of the functional equation of the Riemann zeta function, which reads off as
\[
\dfrac{\zeta'(s)}{\zeta(s)} + \dfrac{\zeta'(1-s)}{\zeta(1-s)} = \log\pi -\dfrac{1}{2}\dfrac{\Gamma'(\frac{s}{2})}{\Gamma(\frac{s}{2})}-\dfrac{1}{2}\dfrac{\Gamma'(\frac{1-s}{2})}{\Gamma(\frac{1-s}{2})}.
\]
Substituting $s = 2it$, and inserting the above identity into our expression for the absolute value of the logarithmic derivative of $\psi(1/2 + it)$, we obtain
\[
\left|\dfrac{\psi'(1/2 + it)}{\psi(1/2 + it)}\right| =
\left|- \dfrac{\Gamma(1/2 + it)'}{\Gamma(1/2 + it)} - \dfrac{\Gamma(1/2 - it)'}{\Gamma(1/2 - it)} - 2\dfrac{\zeta(1 + 2it)'}{\zeta(1 + 2it)} - 2\dfrac{\zeta(1 - 2it)'}{\zeta(1 - 2it)} + 2\log \pi\right|.
\]
Using the bound $|\dfrac{\zeta(1 + 2it)'}{\zeta(1 + 2it)}| = O(\log |t|)$, valid for $|t| \ge 1$, we simplify the above into
\[
\left|\dfrac{\psi'(1/2 + it)}{\psi(1/2 + it)}\right| \ll 1 + \log|t| + \left|\dfrac{\Gamma(1/2 + it)'}{\Gamma(1/2 + it)}\right| \ll 1 + \log(1 + |t|).
\]
Therefore,
\[
|(1)| \le \left|\dfrac{\psi'(1/2 + it)}{\psi(1/2 + it)}\right| \ll 1 + \log(1 + |t|).
\]
We proceed to analyze $(2)$.
\[
(2) = \psi(s)\sum_{p_1|q}\dfrac{d}{ds}(p_1^{2s} - 1)^{-1}\prod_{p|(q/p_1)}(p^{2s} - 1)^{-1}\prod_{\substack{p|\mathfrak{a}\wedge p|\mathfrak{b}\\ \vee \\ p\nmid\mathfrak{a}\wedge p\nmid\mathfrak{b}}}(p-1)
\prod_{\substack{p\mid \mathfrak{a}\wedge p\nmid\mathfrak{b}\\\vee \\
p\mid\mathfrak{b}\wedge p\nmid\mathfrak{a}}}(p^s - p^{1-s}).
\]
Notice that
\[
\dfrac{d}{ds}(p_1^{2s} - 1)^{-1} = -2\log p_1(p_1^{2s} - 1)^{-2}p_1^{2s}
\]
and so we may simplify the above expression into
\[
-\sum_{p|q}\dfrac{p^{2s}\log p}{p^{2s} - 1}\phi_{\mathfrak{a},\mathfrak{b}}(s),
\]
and since $|\phi_{\mathfrak{a},\mathfrak{b}}(1/2 + ir)| \le 1$, substituting back $s = 1/2 + it$, we upper bound $|(2)|$ by
\[
\sum_{p|q}|\dfrac{p^{1 + 2it}\log p}{p^{1 + 2it} - 1}| \le \sum_{p|q}|\dfrac{p\log p}{p - 1}| \ll \log^2 (1 + q).
\]
which implies that $(2) \ll \log^2(1 + q)$. 

We proceed to analyze $(3)$. Recall the definition
\[
(3) = \psi(s)\prod_{p|q}(p^{2s} - 1)^{-1}\prod_{\substack{p|\mathfrak{a}\wedge p|\mathfrak{b}\\ \vee \\ p\nmid\mathfrak{a}\wedge p\nmid\mathfrak{b}}}(p-1)
\sum_{\substack{p_1\mid \mathfrak{a}\wedge p_1\nmid\mathfrak{b}\\\vee \\
p_1\mid\mathfrak{b}\wedge p_1\nmid\mathfrak{a}}}\dfrac{d}{ds}(p_1^s - p_1^{1-s})
\prod_{\substack{p\mid \mathfrak{a}\wedge p\nmid\mathfrak{b}\\\vee \\
p\mid\mathfrak{b}\wedge p\nmid\mathfrak{a}\\ p\neq p_1}}(p^s - p^{1-s}).
\]
Since $\left|\dfrac{p-1}{p^{2s} - 1}\right|,\left|\dfrac{p^s - p^{1-s}}{p^{2s} - 1}\right| \ll 1$, and since $|\psi(s)| = 1$, we upper bound
\[
|(3)| \le \sum_{\substack{p\mid \mathfrak{a}\wedge p\nmid\mathfrak{b}\\\vee \\
p\mid\mathfrak{b}\wedge p\nmid\mathfrak{a}}}\left|\dfrac{\dfrac{d}{ds}(p^s - p^{1-s})}{p^{2s} - 1}\right|.
\]
Taking derivative, we obtain
\[
\dfrac{d}{ds}(p^s - p^{1-s}) = \log p(p^s + p^{1-s}),
\]
and substituting $s = 1/2 + ir$, we may upper bound $|\log p(p^{1/2 + it} + p^{1/2 - it})| \ll \sqrt{p}\log p$. 

Therefore
\[
|(3)| \le \sum_{\substack{p\mid \mathfrak{a}\wedge p\nmid\mathfrak{b}\\\vee \\
p\mid\mathfrak{b}\wedge p\nmid\mathfrak{a}}}\dfrac{\sqrt{p}\log p}{p - 1} \ll \log (1 + q).
\]
Summing the bounds on $|(1)|$, $|(2)|$ and $|(3)|$, we obtain
\[
|\dfrac{d}{dt}\phi_{\mathfrak{a},\mathfrak{b}}(1/2 + it)| \ll \log(1 + |t|) + \log^2(1 + q),
\]
combined with the bound for $|\dfrac{d}{dt}\alpha(n,1/2 + it)|$, we obtain
\[
|\dfrac{d}{dt}\left(\alpha(n,1/2 + it)\phi_{\mathfrak{a},\mathfrak{b}}(1/2 + it)\right)| \ll \log(1 + |t|) + \log^2(1 + q) + \log(1 + |n/2|).
\]
Summing over the set of cusps $\mathfrak{A}$, which has cardinality $O(q^{\epsilon})$, we obtain the bound
\[
\sum_{\mathfrak{b}\in \mathfrak{A}}\left|\dfrac{d}{dt}\left(\alpha(n,1/2 + it)\phi_{\mathfrak{a},\mathfrak{b}}(1/2 + it)\right)\right| \ll_{\epsilon} q^{\epsilon}\left(1 + \log(1 + |t|) + \log(1 + |n/2|)\right),
\]
which completes the proof of Claim~\ref{claim:second_term}.
\end{proof}
Combining the bounds obtained in Claim~\ref{claim:third_term} and Claim~\ref{claim:second_term}, we find that
\[
||\Lambda^TE_{\mathfrak{a},n}(*,1/2 + it)||^2 \ll_{\epsilon} q^{\epsilon}\left(1 + \log(1 + |n/2|) + \log(1 + |t|)\right),
\]
which completes the proof of Claim~\ref{claim:truncated_eisenstein_norm_bound}.
\end{proof}

\section{Bounding the $P$-right shift norm}
In this section we prove Claim~\ref{claim:right_shift_bound}, used in the proof of Theorem~\ref{theorem:sobolev_bound}. Before we are ready to state Claim~\ref{claim:right_shift_bound}, we require the definition of a ``small translate".

\begin{definition}[Small Translate]
We say that an element $g\in SL_2(\mathbb{R})$ is a small translate for $X_0(q)$ if for all $x\in X_0(q)$ one has
\[
h(x)/2 \le h(xg) \le 2h(x),
\]
and if, moreover, $h(x) > 2$, then $x$ and $xg$ are in the same cuspidal zone (see Definition~\ref{definition:same_zone}) for truncation height 1.
\end{definition}
Now that we have defined a property of elements in $SL_2(\mathbb{R})$, we would like to prove that there exist elements which exhibit this property. In fact, we would like to show that there exists a $\delta > 0$ such that the corresponding $\delta$-rectangle, $R(\delta)$, (see Definition~\ref{definition:delta_rectangle}) is a set consisting of small translates.

\begin{claim}
\label{claim:delta_existence}
There exists a $\delta > 0$, such that all $g \in R(\delta)$ are small translates.
\end{claim}
\begin{proof}
We would like to show that there exists a $\delta > 0$ such that for all $x\in X_0(q)$:
\[
\forall g\in R(\delta): h(x)/2 \le h(xg) \le 2h(x),
\]
and that if, moreover, $h(x) > 2$, then $x, xg$ are in the same cuspidal zone.

We reduce our claim to the following statement. We would like to find some $\delta > 0$ such that for all $x\in X_0(q)$, $\sigma_{\mathfrak{a}}$ a scaling matrix and $\gamma\in\Gamma_0(q)$, one has
\[
\forall g\in R(\delta): y(\sigma_{\mathfrak{a}}^{-1}\gamma g_x)/2 \le y(\sigma_{\mathfrak{a}}^{-1}\gamma g_xg) \le 2y(\sigma_{\mathfrak{a}}^{-1}\gamma g_x),
\]
where $g_x$ is an arbitrary representative of $x\in X_0(q)$. We note that under this statement, if $\sigma_{\mathfrak{b}}$ and $\gamma'\in\Gamma$ are such that $h(x) = y(\sigma_{\mathfrak{b}}^{-1}\gamma'g_x) > 2$, then, in particular, $y(\sigma_{\mathfrak{b}}^{-1}\gamma' g_xg) \ge h(x)/2 > 1$, and it follows that both $x$ and $xg$ are in the cuspidal zone of the same cusp (cusp $\mathfrak{b}$).

The rest of our claim also follows from the reduction. As for the upper bound, we note that if 
\[
\max_{\mathfrak{a}}\max_{\gamma\in\Gamma_0(q)}\{y(\sigma_{\mathfrak{a}}^{-1}\gamma g_xg)\}
\]
is attained for the cusp $\mathfrak{b}$ and $\gamma'\in\Gamma_0(q)$, then it follows that
\[
h(xg) \le 2y(\sigma_{\mathfrak{b}}^{-1}\gamma' g_x) \le 2\max_{\mathfrak{a}}\max_{\gamma\in\Gamma_0(q)}\{y(\sigma_{\mathfrak{a}}^{-1}\gamma g_x)\} = 2h(x).
\]
As for the lower bound, if $h(x) = y(\sigma_{\mathfrak{b}^{-1}}\gamma' g_x)$, then
\[
h(x)/2 \le y(\sigma_{\mathfrak{b}^{-1}}\gamma' g_x g) \le \max_{\mathfrak{a}}\max_{\gamma\in\Gamma_0(q)}\{y(\sigma_{\mathfrak{a}}^{-1}\gamma g_x g)\} = h(xg).
\]

Possibly replacing $x$ by $(\sigma_{\mathfrak{a}}^{-1}\gamma)^{-1}x$ without loss of generality, it suffices to find a $\delta > 0$ such that for all $x\in X_0(q)$, one has
\[
\forall g\in R(\delta): y(x)/2 \le y(xg) \le 2y(x).
\]
Let $g_x$ again denote a representative for $x$ in $SL_2(\mathbb{R})$. Since $K$ stabilizes the identity and acts continuously via left translations on $SL_2(\mathbb{R})$, and since for all $\epsilon > 0$, $\text{Int}(R(\epsilon))$ is an open neighborhood of the identity, for each $k$ we have some $\delta_k' > 0$, guaranteed to exist by continuity, such that
\[
k\text{Int}(R(\delta_k'))\subseteq \text{Int}(R(\epsilon)).
\]
We denote by $\delta_k > 0$ the supremum over all positive real numbers with such property. It is clear that the function $\delta : K \longrightarrow \mathbb{R}$ defined by $\delta(k) = \delta_k$ for all $k$ is continuous, and since it is defined over a compact set and attains only positive values, it has a positive minimum, we define it to be $\delta > 0$. Possibly replacing $\delta$ by $\delta / 2 > 0$, we may assume
\[
KR(\delta) \subseteq R(\epsilon).
\]
Next, if $g_x = n_x a_x k_x$ is the Iwasawa decomposition of $g_x$ and if $g\in R(\delta)$, then $g_xg = n_xa_x (k_xg) \in n_xa_x R(\epsilon)$. Now, if $g_{\epsilon} \in R(\epsilon)$ has Iwasawa decomposition $n_{\epsilon} a_{\epsilon} k_{\epsilon}$, then
\[
n_xa_xn_{\epsilon}a_{\epsilon}k_{\epsilon} = n_xn_{\epsilon}'a_xa_{\epsilon}k_{\epsilon},
\]
since $N$ is normal in $NA$. Therefore, $y(g_xg) = y(a_xa_{\epsilon}) = y(x)y(g_{\epsilon})$. If we choose $\epsilon < 1/2$, then 
\[
y(g_x)/2 \le y(g_xg) \le 2y(g_x),
\]
which is what we wanted to show. Since this is true for all $x$, $g$ is a small translate. Since this is true for all $g \in R(\delta)$, we found that there exists a $\delta > 0$, namely $\delta_{\epsilon}$ for some $0 < \epsilon \le 1/2$, such that all $g\in R(\delta)$ are small translates.
\end{proof}

We are now ready to state Claim~\ref{claim:right_shift_bound}.

\begin{claim}
\label{claim:right_shift_bound}
For all small translates $g'\in R(\delta)$ with $\delta > 0$ as in Claim~\ref{claim:delta_existence}, one has
\[
||\rho(g')||_P^2 \ll 1.
\]
where $||\rho(g')||_P = \sup_{v\in V}\dfrac{||\rho(g')v||_P}{||v||_P}$, $V$ is the Eisenstein even principal series (irreducible) representation, and $||\cdot||_P$ is as in Definition~\ref{definition:p_norm}.
\end{claim}

\begin{proof}
As a first step towards the proof of Claim~\ref{claim:right_shift_bound}, we prove the upper bound.
\begin{claim}
\label{claim:bound_on_shifted_p_norm}
For all small translates $g'\in R(\delta)$ with $\delta > 0$ as in Claim~\ref{claim:delta_existence} and $\phi\in V$, one has
\[
||\rho(g')\phi||_P^2 \le ||\Lambda^{T}\phi||^2 + ||\Lambda^{2T}\phi||^2.
\]
\end{claim}
\begin{proof}
Let $g'\in R(\delta)$ and $\phi\in V$ be as in the claim. Let $\mathcal{F}(\Gamma_0(q)\char`\\SL_2(\mathbb{R}))$ be the standard fundamental domain for $X_0(q)$, which we denote by $\mathcal{F}$ in short where there is no confusion with the $q=1$ case.

By definition,
\[
||\rho(g')\phi||_P^2 = \int_{\mathcal{F}}|\Lambda^T(\rho(g')\phi)(g)|^2d\chi(g) = \int_{\mathcal{F}_T}|\Lambda^T(\rho(g')\phi)(g)|^2d\chi(g) + \int_{\mathcal{F}\setminus \mathcal{F}_T}|\Lambda^T(\rho(g')\phi)(g)|^2d\chi(g),
\]
where $\mathcal{F}_T$ is the compact subset of the fundamental domain, defined by
\[
\mathcal{F}_T = \{x\in \mathcal{F}|h(x)\le T\},
\]
where $h(x)$ is the height function, as in Definition~\ref{definition:height_function}. We call $\mathcal{F}_T$ the truncated fundamental domain.

Note that the truncation operator acts as the identity on the truncated fundamental domain, so that
\[
\int_{\mathcal{F}_T}|(\Lambda^T(\rho(g')\phi))(g)|^2d\chi(g) = 
\int_{\mathcal{F}_T}|\phi(gg')|^2d\chi(g).
\]
We would like to find $T' = T'(g',T)$ such that
\[
\int_{\mathcal{F}_T}|\phi(gg')|^2d\chi(g) \le \int_{\mathcal{F}_{T'}}|\phi(g)|^2d\chi(g).
\]
This would follow if we fix $T'\ge T$ such that for all $g$ satisfying $h(g) \le T$ one has also $h(gg') \le T'$.

Since $g'\in R(\delta)$, by Claim~\ref{claim:delta_existence}, for all $g\in SL_2(\mathbb{R})$, one has
\[
h(gg') \le 2h(g).
\]
Implying that it suffices to choose $T' = 2T$.

In this case, we have
\[
\int_{\mathcal{F}_T}|(\Lambda^T(\rho(g')\phi))(g)|^2d\chi(g) \le \int_{\mathcal{F}_{2T}}|\phi(g)|^2d\chi(g) \le \int_{\mathcal{F}}|\Lambda^{2T}\phi(g)|^2d\chi(g) = ||\Lambda^{2T}(\phi)||^2.
\]
Next, assume $h(g) > T$. Since $g'$ is a small translate, $h(gg') \ge T/2$. Since $T > 2$, by Claim~\ref{claim:delta_existence}, $g$ and $gg'$ are in the same cuspidal zone when one decreases the truncation height from $T$ to $T/2$. 

Therefore, 
\[
\int_{\mathcal{F}\setminus \mathcal{F}_T}|(\Lambda^T(\rho(g')\phi))(g)|^2d\chi(g) \le \int_{\mathcal{F}\setminus \mathcal{F}_{T/2}}|(\Lambda^{T/2}(\phi))(g)|^2d\chi(g) \le ||\Lambda^{T/2}\phi||^2,
\]
where the last inequality is obtained by extending the integral over $\mathcal{F}_{T/2}$.

We note that since $||\Lambda^{T}\phi||^2 - ||\Lambda^{T/2}\phi||^2$ is exactly equal to the the integral of the absolute value squared of the constant term on the domain $\mathcal{F}_T\setminus\mathcal{F}_{T/2}$, which is a nonnegative number, we trivially have
\[
||\Lambda^{T/2}\phi||^2 \le ||\Lambda^{T}\phi||^2.
\]
which completes the proof.
\end{proof}

As a next step, we would like to reduce the problem to the analysis of the ratio of the integrals of the constant terms.
\begin{claim}
\label{claim:reduction_to_constant_term_ratio}
For all $g'\in R(\delta)$, with all notations as above, one has
\[
||\rho(g')||_P^2 \ll 1 + \sup_{n\in 2\mathbb{Z}}
\dfrac{\int_1^{T'}\left|1 + \alpha(n,1/2 + it)\phi_{\mathfrak{a},\mathfrak{a}}(1/2 + it)y^{- 2it}\right|^2\dfrac{dy}{y}}{\int_1^{T}\left|1 + \alpha(n,1/2 + it)\phi_{\mathfrak{a},\mathfrak{a}}(1/2 + it)y^{- 2it}\right|^2\dfrac{dy}{y}}.
\]
\end{claim}
\begin{proof}
By the previous claim, we have
\[
\dfrac{||\Lambda^{T}\rho(g')\phi||^2}{||\Lambda^{T}\phi||^2} \le 1 + \dfrac{||\Lambda^{T'}\phi||^2}{||\Lambda^{T}\phi||^2},
\]
where $T\le T' \le 2T$. As a next step, we note that
\[
||\Lambda^{T'}\phi||^2 = ||\Lambda^{T}\phi||^2 + \int_{\mathcal{F}_{T'}\setminus\mathcal{F}_{T}}|a_{\phi,0}(g)|^2d\chi(g),
\]
where $\mathcal{F}_{T'}\setminus\mathcal{F}_{T}$ is the rectangular subdomain of the cuspidal zones between heights $T'$ and $T$, and $a_{\phi,0}(g)$ denotes the constant term of
$\phi$. 

Writing out $\mathcal{F}_{T'}\setminus\mathcal{F}_{T} = \cup_{\mathfrak{b}}R_{\mathfrak{b}}$, where $R_{\mathfrak{b}}$ is the cuspidal zone of the cusp $\mathfrak{b}$ between heights $T$ and $T'$, denoting by $\sigma_{\mathfrak{b}}$ the scaling matrix, sending the cusp at $\infty$ to the cusp at $\mathfrak{b}$, we have
\[
\int_{\mathcal{F}_{T'}\setminus\mathcal{F}_{T}}|a_{\phi,0}(g)|^2d\chi(g) = \sum_{\mathfrak{b}}\int_{R_{\mathfrak{b}}}|a_{\phi,\mathfrak{b},0}(g)|^2d\chi(g),
\]
where $a_{\phi,\mathfrak{b},0}(g)$ is the constant term of $\phi$ when expanded around the cusp $\mathfrak{b}$. We further write,
\[
\sum_{\mathfrak{b}}\int_{\sigma_{\mathfrak{b}}^{-1}R_{\mathfrak{b}}}|a_{\phi,\mathfrak{b},0}(\sigma_{\mathfrak{b}}g)|^2d\chi(g) = 
\int_0^1\int_T^{T'}\int_0^{2\pi}|a_{\phi,\mathfrak{b},0}(\sigma_{\mathfrak{b}}g)|^2\dfrac{dxdyd\theta}{y^2}.
\]
Let $\phi = \sum_{n}\beta_ne_n$, then, since the constant term of $e_n$, when expanded around the cusp $\mathfrak{b}$, is given by the formula
\[
c_{\mathfrak{a},\mathfrak{b},0}(\sigma_{\mathfrak{b}}g) = e^{in\theta}\left(\delta_{\mathfrak{a} = \mathfrak{b}}y^{1/2 + it} + \alpha(n,1/2 + it)\phi_{\mathfrak{a},\mathfrak{b}}(1/2 + it)y^{1/2 - it}\right),
\]
by orthogonality relations in the $\theta$-coordinate, we obtain
\[
\int_0^1\int_T^{T'}\int_0^{2\pi}|a_{\phi,\mathfrak{b},0}(\sigma_{\mathfrak{b}}g)|^2\dfrac{dxdyd\theta}{y^2} = 2\pi\sum_n|\beta_n|^2\int_T^{T'}\left|\delta_{\mathfrak{a} = \mathfrak{b}}y^{1/2 + it} + \alpha(n,1/2 + it)\phi_{\mathfrak{a},\mathfrak{b}}(1/2 + it)y^{1/2 - it}\right|^2\dfrac{dy}{y^2}.
\]
Hence,
\[
\dfrac{||\rho(g')\phi||_P^2}{||\phi||_P^2} \le 2 + \dfrac{2\pi\sum_{\mathfrak{b}}\sum_n|\beta_n|^2\int_T^{T'}\left|\delta_{\mathfrak{a} = \mathfrak{b}}y^{1/2 + it} + \alpha(n,1/2 + it)\phi_{\mathfrak{a},\mathfrak{b}}(1/2 + it)y^{1/2 - it}\right|^2\dfrac{dy}{y^2}}{||\Lambda^T\phi||^2}.
\]
We lower bound the denominator by
\[
||\Lambda^T\phi||^2 = \int_{\mathcal{F}}|(\Lambda^T\phi)(g)|^2d\chi(g) \ge \int_{\mathcal{F}_T\setminus\mathcal{F}_1}|(\Lambda^T\phi)(g)|^2d\chi(g),
\]
and using orthogonality relations in the $x$-coordinate, together with the rectangularity of the domain, following a similar computation for the numerator, we obtain the lower bound
\[
||\Lambda^T\phi||^2\ge 2\pi\sum_{\mathfrak{b}}\sum_n|\beta_n|^2\int_1^{T}\left|\delta_{\mathfrak{a} = \mathfrak{b}}y^{1/2 + it} + \alpha(n,1/2 + it)\phi_{\mathfrak{a},\mathfrak{b}}(1/2 + it)y^{1/2 - it}\right|^2\dfrac{dy}{y^2},
\]
plugging the lower bound back in the ratio while extending the integral in the numerator over the set $[1,T']\supseteq[T,T']$, we obtain the upper bound
\[
\dfrac{||\rho(g')\phi||_P^2}{||\phi||_P^2} \le 2 + \dfrac{\sum_n|\beta_n|^2\sum_{\mathfrak{b}}\int_1^{T'}\left|\delta_{\mathfrak{a} = \mathfrak{b}}y^{1/2 + it} + \alpha(n,1/2 + it)\phi_{\mathfrak{a},\mathfrak{b}}(1/2 + it)y^{1/2 - it}\right|^2\dfrac{dy}{y^2}}{\sum_n|\beta_n|^2\sum_{\mathfrak{b}}\int_1^{T}\left|\delta_{\mathfrak{a} = \mathfrak{b}}y^{1/2 + it} + \alpha(n,1/2 + it)\phi_{\mathfrak{a},\mathfrak{b}}(1/2 + it)y^{1/2 - it}\right|^2\dfrac{dy}{y^2}}.
\]
A simple application of the mediant inequality shows that
\[
\dfrac{\sum_n|\beta_n|^2\sum_{\mathfrak{b}}\int_1^{T'}\left|\delta_{\mathfrak{a} = \mathfrak{b}}y^{1/2 + it} + \alpha(n,1/2 + it)\phi_{\mathfrak{a},\mathfrak{b}}(1/2 + it)y^{1/2 - it}\right|^2\dfrac{dy}{y^2}}{\sum_n|\beta_n|^2\sum_{\mathfrak{b}}\int_1^{T}\left|\delta_{\mathfrak{a} = \mathfrak{b}}y^{1/2 + it} + \alpha(n,1/2 + it)\phi_{\mathfrak{a},\mathfrak{b}}(1/2 + it)y^{1/2 - it}\right|^2\dfrac{dy}{y^2}}
\]
\[
\le \sup_{n\in 2\mathbb{Z}: \beta_n\neq 0}
\dfrac{\sum_{\mathfrak{b}}\int_1^{T'}\left|\delta_{\mathfrak{a} = \mathfrak{b}}y^{1/2 + it} + \alpha(n,1/2 + it)\phi_{\mathfrak{a},\mathfrak{b}}(1/2 + it)y^{1/2 - it}\right|^2\dfrac{dy}{y^2}}{\sum_{\mathfrak{b}}\int_1^{T}\left|\delta_{\mathfrak{a} = \mathfrak{b}}y^{1/2 + it} + \alpha(n,1/2 + it)\phi_{\mathfrak{a},\mathfrak{b}}(1/2 + it)y^{1/2 - it}\right|^2\dfrac{dy}{y^2}}
\]
\[
\sup_{n\in 2\mathbb{Z}}
\dfrac{\sum_{\mathfrak{b}}\int_1^{T'}\left|\delta_{\mathfrak{a} = \mathfrak{b}}y^{1/2 + it} + \alpha(n,1/2 + it)\phi_{\mathfrak{a},\mathfrak{b}}(1/2 + it)y^{1/2 - it}\right|^2\dfrac{dy}{y^2}}{\sum_{\mathfrak{b}}\int_1^{T}\left|\delta_{\mathfrak{a} = \mathfrak{b}}y^{1/2 + it} + \alpha(n,1/2 + it)\phi_{\mathfrak{a},\mathfrak{b}}(1/2 + it)y^{1/2 - it}\right|^2\dfrac{dy}{y^2}}.
\]
Note that
\[
\sum_{\mathfrak{b}}\int_1^{U}\left|\delta_{\mathfrak{a} = \mathfrak{b}}y^{1/2 + it} + \alpha(n,1/2 + it)\phi_{\mathfrak{a},\mathfrak{b}}(1/2 + it)y^{1/2 - it}\right|^2\dfrac{dy}{y^2} = 
\]
\[
\int_1^{U}|1 + y^{-2it}\alpha(n,1/2 + it)\phi_{\mathfrak{a},\mathfrak{a}}(1/2 + it)|^2\dfrac{dy}{y} + \sum_{\mathfrak{b}\neq\mathfrak{a}}\int_{1}^{U}|\phi_{\mathfrak{a},\mathfrak{b}}(1/2 + it)|^2\dfrac{dy}{y},
\]
and because of the unitarity of $\Phi(1/2 + it)$, the second term is simplified into
\[
\log U\sum_{\mathfrak{b}\neq\mathfrak{a}}|\phi_{\mathfrak{a},\mathfrak{b}}(1/2 + it)|^2 = \log U\left(1 - |\phi_{\mathfrak{a},\mathfrak{a}}(1/2 + it)|^2\right).
\]
Plugging back into the numerator and the denominator, while applying the mediant inequality, we obtain
\[
\dfrac{||\rho(g')\phi||_P^2}{||\phi||_P^2} \le 2 + \dfrac{\log T'}{\log T} + \sup_{n\in 2\mathbb{Z}}
\dfrac{\int_1^{T'}\left|1 + \alpha(n,1/2 + it)\phi_{\mathfrak{a},\mathfrak{a}}(1/2 + it)y^{- 2it}\right|^2\dfrac{dy}{y}}{\int_1^{T}\left|1 + \alpha(n,1/2 + it)\phi_{\mathfrak{a},\mathfrak{a}}(1/2 + it)y^{- 2it}\right|^2\dfrac{dy}{y}},
\]
which completes the proof since $T \gg 1$ and $T' \ll 1$.
\end{proof}
\begin{claim}
\label{claim:constant_term_ratio}
One has for all $n\in 2\mathbb{Z}$,
\[
\dfrac{\int_1^{T'}|1 + y^{-2it}\alpha(n,1/2 + it)\phi_{\mathfrak{a},\mathfrak{a}}(1/2 + it)|^2\dfrac{dy}{y}}{\int_1^{T}|1 + y^{-2it}\alpha(n,1/2 + it)\phi_{\mathfrak{a},\mathfrak{a}}(1/2 + it)|^2\dfrac{dy}{y}} \ll 1.
\]
\end{claim}
\begin{proof}
This is proven by case analysis. Assume first that $|\phi_{\mathfrak{a},\mathfrak{a}}(1/2 + it)| < 1/2$. In this case, by the triangle inequality we find that
\[
1/2 \le |1 + y^{-2it}\alpha(n,1/2 + it)\phi_{\mathfrak{a},\mathfrak{a}}(1/2 + it)| \le 3/2,
\]
which implies that
\[
\dfrac{\int_{1}^{T'}|1 + y^{-2it}\alpha(n,1/2 + it)\phi_{\mathfrak{a},\mathfrak{a}}(1/2 + it)|^2\dfrac{dy}{y}}{\int_1^{T}|1 + y^{-2it}\alpha(n,1/2 + it)\phi_{\mathfrak{a},\mathfrak{a}}(1/2 + it)|^2\dfrac{dy}{y}} \ll \dfrac{\log T'}{\log T} \ll 1,
\]
and the claim follows. Therefore, we may assume throughout the rest of the proof that $|\phi_{\mathfrak{a},\mathfrak{a}}(1/2 + it)| \ge 1/2$. We analyze this case separately for small and large values of $|t|$, starting from the simpler case, which is the case where $|t| > 1$. 

\begin{claim}
\label{claim:large_t_bound}
Assume $|\phi_{\mathfrak{a},\mathfrak{a}}(1/2 + it)| \ge 1/2$ and $|t| > 1$, then
\[
\int_{1}^{T}|1 + y^{-2it}\alpha(n,1/2 + it)\phi_{\mathfrak{a},\mathfrak{a}}(1/2 + it)|^2\dfrac{dy}{y} \gg 1.
\]
\end{claim}
\begin{proof}
Denote by $y_0\in [1,\infty)$ a real number for which $1/2 \le \alpha(1/2 + it,n)\phi_{\mathfrak{a},\mathfrak{a}}(1/2 + it)y_0^{-2it} \in \mathbb{R}$. For such $y_0$,
\[
|1 + y_0^{-2it}\alpha(n,1/2 + it)\phi_{\mathfrak{a},\mathfrak{a}}(1/2 + it)|^2 \ge 9/4.
\]
Possibly replacing $y_0\in [1,\infty)$ by some larger value, we may assume that there exists an interval around $y_0$ within $[1,\infty)$ with boundaries $y_1,y_2$, and the property that for all $y\in [y_1,y_2]$, one has 
\[
|\arg(y^{-2it}\alpha(n,1/2 + it)\phi_{\mathfrak{a},\mathfrak{a}}(1/2 + it))| \le \dfrac{\pi}{2}.
\]
For all such $y$, one has
\[
|1 + y^{-2it}\alpha(n,1/2 + it)\psi(1/2 + it)|^2 \ge 1.
\]
We may trivially lower bound
\[
\int_{y_1}^{y_2}|1 + y^{-2it}\alpha(n,1/2 + it)\phi_{\mathfrak{a},\mathfrak{a}}(1/2 + it)|^2\dfrac{dy}{y} \ge
\int_{y_1}^{y_2}\dfrac{dy}{y} = \log\dfrac{y_2}{y_1}.
\]
Next, we would like to estimate $\log\dfrac{y_2}{y_1}$. By assumption,
\[
\arg(y_2^{-2it}\alpha(n,1/2 + it)\phi_{\mathfrak{a},\mathfrak{a}}(1/2 + it)) - \arg(y_1^{-2it}\alpha(n,1/2 + it)\phi_{\mathfrak{a},\mathfrak{a}}(1/2 + it)) = \pi,
\]
while, in fact,
\[
|\arg(y_2^{-2it}\alpha(n,1/2 + it)\phi_{\mathfrak{a},\mathfrak{a}}(1/2 + it)) - \arg(y_1^{-2it}\alpha(n,1/2 + it)\phi_{\mathfrak{a},\mathfrak{a}}(1/2 + it))| = |2t\log \dfrac{y_2}{y_1}|.
\]
Therefore,
\[
\log \dfrac{y_2}{y_1} = \dfrac{\pi}{2|t|}.
\]
Plugging this back into $\int_{y_1}^{y_2}|1 + y^{-2it}\alpha(n,1/2 + it)\phi_{\mathfrak{a},\mathfrak{a}}(1/2 + it)|^2\dfrac{dy}{y}$, we obtain the lower bound
\[
\int_{y_1}^{y_2}|1 + y^{-2it}\alpha(n,1/2 + it)\phi_{\mathfrak{a},\mathfrak{a}}(1/2 + it)|^2\dfrac{dy}{y} \ge \dfrac{\pi}{2|t|}.
\]
Next, we note that a sufficient condition to have $[t]$-disjoint intervals $[y_1^i,y_2^i]_{i=1}^{[t]}\subseteq[1,T]$, is that when $y$ traverses through the interval $[1,T]$, $y^{-2it}$ traverses an arc of length at least $3\pi[t]$ on the unit circle. In this case, for each $1\le i\le [t]$, we have a unique interval $I_{3\pi}^i$ in which $y^{-2it}$ traverses an arc of length $3\pi$. For each interval $I_{3\pi}^i$ there exists a subinterval $I_{2\pi}^i$ in which $y^{-2it}$ traverses an arclength of $2\pi$, with the property that for each $y\in I_{2\pi}^i$ there exists an interval $[y-\delta_1,y+\delta_2]\subseteq I_{3\pi}^i$, such that $y^{-2it}$ traverses the arc $[\arg(y^{-2it})-\pi/2,\arg(y^{-2it})+\pi/2]$. Therefore, there exists $y_0^i\in I_{2\pi}^i\subseteq I_{3\pi}^i$ such that $\arg((y_0^i)^{-2it}) = \theta_0$, and a subinterval $[y_1^i,y_2^i]\subseteq I_{3\pi}^i$ such that $y^{-2it}$ obtains all arguments in $[\theta_0-\pi/2,\theta_0+\pi/2]$. 

Next, we claim that a sufficient condition to having $y^{-2it}$ traverse an arc of length $3\pi[t]$ as $y$ traverses $[1,T]$, it is sufficient to have $T$ large enough such that
\[
|-2t\log T|\ge 3\pi[t] \iff T\ge e^{\frac{3\pi[t]}{2|t|}},
\]
and since by assumption $T \ge e^{2\pi}$, the inequality holds.

Therefore, since
\[
\int_{1}^{T}|1 + y^{-2it}\alpha(n,1/2 + it)\phi_{\mathfrak{a},\mathfrak{a}}(1/2 + it)|^2\dfrac{dy}{y} \ge 
\sum_{j=1}^{[t]}\int_{y_1^j}^{y_2^{j}}|1 + y^{-2it}\alpha(n,1/2 + it)\phi_{\mathfrak{a},\mathfrak{a}}(1/2 + it)|^2\dfrac{dy}{y} \ge \sum_{j=1}^{[t]}\dfrac{\pi}{2|t|} \gg 1,
\]
our claim holds.
\end{proof}

We find that
\[
\dfrac{\int_{1}^{T'}|1 + y^{-2it}\alpha(n,1/2 + it)\phi_{\mathfrak{a},\mathfrak{a}}(1/2 + it)|^2\dfrac{dy}{y}}{\int_1^{T}|1 + y^{-2it}\alpha(n,1/2 + it)\phi_{\mathfrak{a},\mathfrak{a}}(1/2 + it)|^2\dfrac{dy}{y}} \ll \log T' \ll 1,
\]
as $\int_{1}^{T'}|1 + y^{-2it}\alpha(n,1/2 + it)\phi_{\mathfrak{a},\mathfrak{a}}(1/2 + it)|^2\dfrac{dy}{y} \ll \log T' \ll 1$, trivially.

Next, we assume that $0 < |t|\le 1$. Our first objective will be to bound the supremum of $|1 + \alpha(1/2 + it,n)\phi_{\mathfrak{a},\mathfrak{a}}(1/2 + it)y^{-2it}|$ from below for $y\in [1,e]\subseteq [1,T],[1,T']$. 

\begin{claim}
\label{claim:integrand_sup_bound}
Let $0 < |t|\le 1$, $y\in [1,e]$ and assume $|\phi_{\mathfrak{a},\mathfrak{a}}(1/2 + it)| \ge 1/2$, then
\[
\sup_{y\in [1,e]}\left|1 + \alpha(1/2 + it,n)\phi_{\mathfrak{a},\mathfrak{a}}(1/2 + it)y^{-2it}\right| \gg |t|.
\]
\end{claim}
\begin{proof}
We write
\[
|1 + \alpha(1/2 + it,n)\phi_{\mathfrak{a},\mathfrak{a}}(1/2 + it)y^{-2it}| = 
\]
\[
\left|1 + \alpha(1/2 + it,n)\phi_{\mathfrak{a},\mathfrak{a}}(1/2 + it) - \alpha(1/2 + it,n)\phi_{\mathfrak{a},\mathfrak{a}}(1/2 + it) + \alpha(1/2 + it,n)\phi_{\mathfrak{a},\mathfrak{a}}(1/2 + it)y^{-2it}\right|
\]
\[
\ge
|\alpha(1/2 + it,n)\phi_{\mathfrak{a},\mathfrak{a}}(1/2 + it)||y^{-2it} - 1| - |1 + \alpha(1/2 + it,n)\phi_{\mathfrak{a},\mathfrak{a}}1/2 + it)|
\]
\[
= |\phi_{\mathfrak{a},\mathfrak{a}}(1/2 + it)||y^{-2it} - 1| - |1 + \alpha(1/2 + it,n)\phi_{\mathfrak{a},\mathfrak{a}}(1/2 + it)|
\]
\[
\ge |y^{-2it} - 1|/2 - |1 + \alpha(1/2 + it,n)\phi_{\mathfrak{a},\mathfrak{a}}(1/2 + it)|
\]
since $|\alpha(1/2 + it,n)| = 1$ and by assumption, $|\phi_{\mathfrak{a},\mathfrak{a}}(1/2 + it)| \ge 1/2$. Now, either $|1 + \alpha(1/2 + it,n)\phi_{\mathfrak{a},\mathfrak{a}}(1/2 + it)| > |t|/2$, and then the supremum is too (substitute $y=1$), or the above expression is at least
\[
|y^{-2it} - 1|/2 - |t|/2.
\]
Observe that
\[
|y^{-2it} - 1| = \sqrt{2 - y^{-2it} - y^{2it}} = \sqrt{2}\sqrt{1 - \cos(2t\log y)} = 2\sin(|t|\log y).
\]
Assuming $y\in [1,e]$, then the Taylor expansion of $\sin(|t|\log y)$ in $|t|\log y$ is a Leibniz series, and is lower bounded by
\[
\sin(|t|\log y) \ge |t|\log y - \dfrac{(|t|\log y)^3}{6} \ge \dfrac{5|t|\log y}{6},
\]
so that for $y\in [1,e]$, one has
\[
|y^{-2it} - 1|/2 \ge 5(|t|\log y)/6.
\]
In particular, for $y=e$, we find that
\[
|y^{-2it} - 1|/2 - |t|/2 \ge |t|/3,
\]
so that in any case $\sup_{y\in [1,e]}|1 + \alpha(1/2 + it,n)\phi_{\mathfrak{a},\mathfrak{a}}(1/2 + it)y^{-2it}| \ge |t|/3$, which is what we wanted to show.
\end{proof}
Next, we notice that there is a relationship between $\sup_{y\in [1,V]}|1 + \alpha(1/2 + it,n)\phi_{\mathfrak{a},\mathfrak{a}}(1/2 + it)y^{-2it}|$ and $\int_1^V|1 + \alpha(1/2 + it,n)\phi_{\mathfrak{a},\mathfrak{a}}(1/2 + it)y^{-2it}|^2\dfrac{dy}{y}$, whenever $V \ge e$.

\begin{claim}
\label{claim:s_and_I_relations}
Denote by $s(V)$ the supremum $\sup_{y\in [1,V]}|1 + \alpha(1/2 + it,n)\phi_{\mathfrak{a},\mathfrak{a}}(1/2 + it)y^{-2it}|$, with $V \ge e$, and by $I(V)$ the value of the integral $\int_1^V|1 + \alpha(1/2 + it,n)\phi_{\mathfrak{a},\mathfrak{a}}(1/2 + it)y^{-2it}|^2\dfrac{dy}{y}$. Then,
\begin{itemize}
    \item $I(V) \le s^2(V)\log V$,
    \item $I(V) \gg s^2(V)$.
\end{itemize}
\end{claim}
\begin{proof}
The upper bound is trivial. As for the lower bound, let $y_0\in [1,V]$ be the point in which $|1 + \alpha(1/2 + it,n)\phi_{\mathfrak{a},\mathfrak{a}}(1/2 + it)y^{-2it}|$ attains its supremum. Consider the ratio
\[
\dfrac{|1 + \alpha(1/2 + it,n)\phi_{\mathfrak{a},\mathfrak{a}}(1/2 + it)y^{-2it}|}{|1 + \alpha(1/2 + it,n)\phi_{\mathfrak{a},\mathfrak{a}}(1/2 + it)y_0^{-2it}|} =
|1 + \dfrac{\alpha(1/2 + it,n)\phi_{\mathfrak{a},\mathfrak{a}}(1/2 + it)y^{-2it} - \alpha(1/2 + it,n)\phi_{\mathfrak{a},\mathfrak{a}}(1/2 + it)y_0^{-2it}}{1 + \alpha(1/2 + it,n)\phi_{\mathfrak{a},\mathfrak{a}}(1/2 + it)y_0^{-2it}}|
\]
\[
\ge 1 - \dfrac{|y^{-2it} - y_0^{-2it}|}{s(V)}.
\]
Clearly, if $|t\log\dfrac{y}{y_0}| \le 1$, then
\[
|y^{-2it} - y_0^{-2it}| = |(y/y_0)^{-2it} - 1| \ll |t|\cdot|\log y/y_0|,
\]
combining with our knowledge from the previous claim that $s(V) \gg |t|$, we find that 
\[
\dfrac{|y^{-2it} - y_0^{-2it}|}{s(V)} \ll |\log y/y_0|,
\]
so that there exists some constant $e > c > 1$ (very close to 1), such that for all $y\in (1,\infty)$ with $1/c \le y/y_0 \le c$, one has
\[
\dfrac{|1 + \alpha(1/2 + it,n)\phi_{\mathfrak{a},\mathfrak{a}}(1/2 + it)y^{-2it}|}{|1 + \alpha(1/2 + it,n)\phi_{\mathfrak{a},\mathfrak{a}}(1/2 + it)y_0^{-2it}|} \ge 1/2.
\]
It follows that since $V \ge e > c$, one has
\[
I(V) \gg s^2(V).
\]
\end{proof}
Applying the above Claim for $V = T$, and denoting $s = s(T)$, $I = I(T)$, we obtain the inequality $I \gg s^2$. Similarly, writing $s' = s(T')$ and $I' = I(T')$, we find that
\[
(s')^2 \ll I' \le (s')^2\log T' \ll (s')^2.
\]
We will only need the trivial implication $I' \ll (s')^2$.

Let $y\in [1,T']$ be arbitrary, and denote by $y_0\in [1,T]$ the $y$ for which $|1 + \alpha(1/2 + it,n)\phi_{\mathfrak{a},\mathfrak{a}}(1/2 + it)y^{-2it}|$ attains its supremum on the interval $[1,T]$. Then
\[
\dfrac{|1 + \alpha(1/2 + it,n)\phi_{\mathfrak{a},\mathfrak{a}}(1/2 + it)y^{-2it}|}{|1 + \alpha(1/2 + it,n)\phi_{\mathfrak{a},\mathfrak{a}}(1/2 + it)y_0^{-2it}|} \ll 1 + \dfrac{|t|\cdot|\log y/y_0|}{s} \ll 1 + |\log y/y_0| \ll 1 + \log T + \log T' \ll 1,
\]
where again we used the fact that $|s|\gg |t|$.

Since this holds for all $y\in [1,T']$, we conclude that
$s' \ll s$, so that $I' \ll s^2$. Applying the bound $I \gg s^2$, we conclude that
\[
\dfrac{\int_{1}^{T'}|1 + y^{-2it}\alpha(n,1/2 + it)\phi_{\mathfrak{a},\mathfrak{a}}(1/2 + it)|^2\dfrac{dy}{y}}{\int_1^{T}|1 + y^{-2it}\alpha(n,1/2 + it)\phi_{\mathfrak{a},\mathfrak{a}}(1/2 + it)|^2\dfrac{dy}{y}} = \dfrac{I'}{I} \ll 1.
\]
\end{proof}
Finally, the case where $t = 0$ is trivial, since then $|1 + y^{-2it}\alpha(n,1/2 + it)\phi_{\mathfrak{a},\mathfrak{a}}(1/2)|^2 = 4$ for all $y\in [1,T]$ and the analysis follows in a similar way as that in the case where $|\phi_{\mathfrak{a},\mathfrak{a}}(1/2 + it)| \le 1/2$.

Combining the different bounds for $|t| = 0$, $0 < |t| \le 1$ and $|t| > 1$, we obtain
\[
\dfrac{\int_{1}^{T'}|1 + y^{-2it}\alpha(n,1/2 + it)\phi_{\mathfrak{a},\mathfrak{a}}(1/2 + it)|^2\dfrac{dy}{y}}{\int_1^{T}|1 + y^{-2it}\alpha(n,1/2 + it)\phi_{\mathfrak{a},\mathfrak{a}}(1/2 + it)|^2\dfrac{dy}{y}} \ll 1.
\]
Claim~\ref{claim:constant_term_ratio} together with Claim~\ref{claim:reduction_to_constant_term_ratio} imply that for all $g'\in R(\delta)$, one has
\[
||\rho(g')||_P^2 \ll 1,
\]
which completes the proof of Claim~\ref{claim:right_shift_bound}.
\end{proof}

\section{Bounding the $R$-right shift norm}
Our goal in this section is to bound the $R$-right shift norm, $||\rho(g)||_R$ (see Corollary~\ref{corollary:right_shift_r_norm}). We approach this task by comparing $R$ with closely related norm, $N$, defined in terms of the Lie algebra, for which the right shift action can be bounded in terms of the norm of the adjoint action of $g$ on the Lie algebra and $||\rho(g)||_P$.

Let $\mathfrak{g} = \mathfrak{sl}_2(\mathbb{R})$ denote the Lie algebra of $SL_2(\mathbb{R})$, and consider the orthonormal basis $U,W,H$ of the standard $K$-invariant Hermitian bilinear form $(X,Y)\mapsto tr(XY^*)$,
\[
\sqrt{2}U = u =
\begin{pmatrix} 
0 & 1 \\
1 & 0 
\end{pmatrix}
\quad
\sqrt{2}W = \kappa = 
\begin{pmatrix} 
0 & 1 \\
-1 & 0 
\end{pmatrix}
\quad
\sqrt{2}\hat{H} = \hat{h} = 
\begin{pmatrix} 
1 & 0 \\
0 & -1 
\end{pmatrix}.
\]
The elements of $\mathfrak{g}$ act on the subspace of smooth functions on $X_0(q)$ as follows. Let $M\in \mathfrak{g}$, and $\phi\in C^{\infty}(X)$ be a smooth function. We write
\[
M\phi(g) = \dfrac{d}{dt}\left(\phi(ge^{tM})\right)|_{t=0}.
\]
We define the norm $N$ via the basis $U,W,\hat{H}$. 
\begin{definition}[The norm $N$]
Given an Eisenstein series $f\in V$, we define its $N$ norm by
\[
||f||_N^2 = N(f) := P(Uf) + P(Wf) + P(\hat{H}f).
\]
\end{definition}

The following claim shows that the norm $N$ behaves nicely under right translations.
\begin{claim}
\label{claim:norm_n_right_shift_bound}
For all $g\in SL_2(\mathbb{R})$, one has
\[
||\rho(g)||_N^2 \ll ||\rho(g)||_P^2||Ad(g)||^2.
\]
\end{claim}
\begin{proof}
Let $X\in\mathfrak{g}$ be an element of the Lie algebra, and $v\in V$ be an Eisenstein series, then
\[
P(X\rho(g)v) = ||X\rho(g)v||_P^2 = ||\rho(g)\rho(gXg^{-1})v||_P^2,
\]
therefore,
\[
\dfrac{||\rho(g)v||_N^2}{||v||_N^2} = 
\dfrac{||\rho(g)\rho(gUg^{-1})v||_P^2 + ||\rho(g)\rho(gWg^{-1})v||_P^2 + ||\rho(g)\rho(g\hat{H}g^{-1})v||_P^2}{||\rho(U)v||_P^2 + ||\rho(W)v||_P^2 + ||\rho(\hat{H})v||_P^2} \ll ||\rho(g)||_P^2||Ad(g)||^2.
\]
Here $||\rho(g)||_P^2$ is given by $\sup_{v\in V}\dfrac{||\rho(g)v||_P^2}{||v||_P}$, and $||Ad(g)||^2$ is the norm of the adjoint action of $g$ on the Lie algebra $\mathfrak{sl}_2(\mathbb{R})$.

Taking the supremum over $0\neq v\in V$ on the left hand side, the claim follows.
\end{proof}

The following Claim gives an upper bound on $||\rho(g)||_R^2$ in terms of $||\rho(g)||_P^2$ and $||\rho(g)||_N^2$.
\begin{claim}
\label{claim:q_norm_right_shift_bound}
Let $g\in SL_2(\mathbb{R})$. One has
\[
||\rho(g)||_R^2 \ll_{\epsilon} q^{\epsilon}(||\rho(g)||_P^2 + ||\rho(g)||_N^2),\quad \forall\epsilon > 0.
\]
\end{claim}
\begin{proof}
Let $f\in V$ be an Eisenstein series, and let $\epsilon > 0$. By Claim~\ref{claim:norms_Q_and_N_diff}, one has
\[
R(\rho(g)f) \ll_{\epsilon} q^{\epsilon}(N(\rho(g)f) + P(\rho(g)f)).
\]
Therefore,
\[
R(\rho(g)f) \ll_{\epsilon} q^{\epsilon}(||\rho(g)||_N^2N(f) + ||\rho(g)||_P^2P(f)).
\]
Decomposing $f = \sum_n\alpha_ne_n$ and using the fact that $e_n$ is an orthogonal basis relative to the norm $P$, we have
\[
P(f) = \sum_n|\alpha_n|^2P(e_n) \le \sum_n|\alpha_n|^2(2\lambda + n^2)P(e_n) = R(f).
\]
Next, either $N(f) \le R(f)$ as well, or $N(f) > R(f)$, and therefore, by Claim~\ref{claim:norms_Q_and_N_diff_pure}, applied with $g = I$,
\[
N(f) \ll_{\epsilon} R(f) + q^{\epsilon}P(f)^{3/8}R(f)^{5/8}.
\]
Hence, $N(f) \ll_{\epsilon} q^{\epsilon}R(f)$, and thus
\[
R(\rho(g)f) \ll_{\epsilon} q^{\epsilon}(||\rho(g)||_N^2q^{\epsilon} + ||\rho(g)||_P^2)R(f),
\]
which completes the proof.
\end{proof}

\begin{corollary}[$R$-norm right shift bound]
\label{corollary:right_shift_r_norm}
For all $g\in SL_2(\mathbb{R})$, one has
\[
||\rho(g)||_R^2 \ll_{\epsilon} q^{\epsilon}(||Ad(g)||^2 + 1)||\rho(g)||_P^2, \quad \forall \epsilon > 0.
\]
\end{corollary}
\begin{proof}
This is an immediate consequence of Claim~\ref{claim:q_norm_right_shift_bound} and Claim~\ref{claim:norm_n_right_shift_bound}.
\end{proof}

\begin{claim}
\label{claim:norms_Q_and_N_diff}
Let $f\in V$ be an Eisenstein series and let $g\in SL_2(\mathbb{R})$. One has
\[
R(\rho(g)f) \ll_{\epsilon}
q^{\epsilon}\left(N(\rho(g)f) + P(\rho(g)f)\right),\quad \forall \epsilon > 0.
\]
\end{claim}
\begin{proof}
Either $R(\rho(g)f) \le N(\rho(g)f)$ and we are done, or $R(\rho(g)f) > N(\rho(g)f)$. In this case, by Claim~\ref{claim:norms_Q_and_N_diff_pure}, we have
\[
R(\rho(g)f) \ll_{\epsilon} N(\rho(g)f) + q^{\epsilon}P(\rho(g)f)^{3/8}R(\rho(g)f)^{5/8}.
\]
Hence,
\[
R(\rho(g)f)^{3/8} \ll_{\epsilon} \dfrac{N(\rho(g)f)}{R(\rho(g)f)^{5/8}} + q^{\epsilon}P(\rho(g)f)^{3/8} \ll_{\epsilon} q^{\epsilon}\left(N(\rho(g)f)^{3/8} + P(\rho(g)f)^{3/8}\right),
\]
and therefore
\[
R(\rho(g)f) \ll_{\epsilon}
q^{\epsilon}\left(N(\rho(g)f) + P(\rho(g)f)\right).
\]
\end{proof}

\begin{claim}
\label{claim:norms_Q_and_N_diff_pure}
Let $f\in V$ be an Eisenstein series and let $g\in SL_2(\mathbb{R})$. One has
\[
\left|R(\rho(g)f) - N(\rho(g)f)\right| \ll_{\epsilon} q^{\epsilon}P(\rho(g)f)^{3/8}R(\rho(g)f)^{5/8}, \quad \forall 0 < \epsilon.
\]
\end{claim}
\begin{proof}
To see that, we write out the definition of $N(\rho(g)f)$.
\[
N(\rho(g)f) = P(U\rho(g)f) + P(W\rho(g)f) + P(\hat{H}\rho(g)f).
\]
The group $SL_2(\mathbb{R})$ is acting transitively on $X_0(q)$ (from the right). We choose the basis $U,W,\hat{H}$ on the Lie algebra of $SL_2(\mathbb{R})$. The action makes a canonical identification between the Lie algebra and the tangent space at each point on the manifold, and therefore it induces a frame on the manifold. There is a unique Riemanian structure on $X_0(q)$ such that this frame is orthonormal.

The induced gradient is given by $\Tilde{\nabla} f = (Uf,Wf,\hat{H}f)^t$, and the induced Laplacian, denoted by $\Tilde{\Delta}$, and defined to be $\text{Div}\circ \Tilde{\nabla}$, is given by $\Tilde{\Delta} f = (U^2 + W^2 + \hat{H}^2)f$.
Now, given a differential operator $X$, we note that
\[
P(X\rho(g)f) = \int_{X_0(q)}|\Lambda^TX\rho(g)f|^2(x)d\chi(x) = \int_{X_0(q)^T}|\Lambda^TX\rho(g)f|^2(x)d\chi(x) + \int_{X_0(q)^C}|\Lambda^TX\rho(g)f|^2(x)d\chi(x) = (I)_X + (II)_X,
\]
where $X_0(q)^T$ is the truncated modular 3-fold (points with height at most $T$) and $X_0(q)^C$ are the cuspidal zones (the complement). Then, since $\Lambda^T$ acts as the identity on $X_0(q)^T$, we have
\[
(I)_U + (I)_W + (I)_{\hat{H}} = \int_{X_0(q)^T}|U\rho(g)f|^2(x)d\chi(x) + \int_{X_0(q)^T}|W\rho(g)f|^2(x)d\chi(x) + \int_{X_0(q)^T}|\hat{H}\rho(g)f|^2(x)d\chi(x)
\]
\[
= \int_{X_0(q)^T}<(\Tilde{\nabla}\rho(g)f)(x),(\Tilde{\nabla}\rho(g)f)(x)>d\chi(x).
\]
Denote by $X$ a Lie derivative, then except for a measure zero set (the boundary of $X_0(q)^C$), one has $\Lambda^TX = X\Lambda^T$. Next, recall that on $X_0(q)^C$, $\rho(g)$ and the constant term operator commute, i.e. if for $v\in V$, we denote by $c_p(v)$ the constant term of $v$, then for all $x\in X_0(q)^C$, one has $(\Lambda^T\rho(g)e_n)(x) = \rho(g)e_n(x) - \rho(g)c_p(f)(x)$. Therefore,
\[
(II)_U + (II)_W + (II)_{\hat{H}} = \int_{X_0(q)^C}|U(\rho(g)f - \rho(g)c_p(f))(x)|^2d\chi(x) + \int_{X_0(q)^C}|W(\rho(g)f - \rho(g)c_p(f))(x)|^2d\chi(x)
\]
\[
+ \int_{X_0(q)^C}|\hat{H}(\rho(g)f - \rho(g)c_p(f))|^2(x)d\chi(x)
= \int_{X_0(q)^C}<(\Tilde{\nabla}(\rho(g)f - \rho(g)c_p(f)))(x),(\Tilde{\nabla}(\rho(g)f - \rho(g)c_p(f)))(x)>d\chi(x).
\]
Since the domain $X_0(q)^C$ is rectangular in the Iwasawa $x$-coordinate, we have orthogonality of characters, which implies that the above expression can be further simplified into
\[
(II)_U + (II)_W + (II)_{\hat{H}} = 
\int_{X_0(q)^C}<(\Tilde{\nabla}(\rho(g)f - \rho(g)c_p(f)))(x),(\Tilde{\nabla}\rho(g)f)(x)>d\chi(x).
\]
By Green's identity, applied for the domain $X_0(q)^T$, we have
\[
\int_{X_0(q)^T}<(\Tilde{\nabla}\rho(g)f)(x),(\Tilde{\nabla}\rho(g)f)(x)>d\chi(x)
+ \int_{X_0(q)^T}(\Tilde{\Delta}\rho(g)f)(x)\overline{\rho(g)f}d\chi(x) = 
\int_{\partial X_0(q)^T}((\Tilde{\nabla}\rho(g)f)(x)\cdot \vec{n}(x))\overline{\rho(g)f(x)}d\chi(x).
\]
Applying a similar identity for $X_0(q)^C$ and summing, we find that
\[
N(\rho(g)f) = -\int_{X_0(q)}(\Lambda^T\Tilde{\Delta}\rho(g)f)(x)\overline{\Lambda^T\rho(g)f(x)}d\chi(x) + 
\]
\[
\int_{\partial X_0(q)^T}((\Tilde{\nabla}\rho(g)f)(x)\cdot \vec{n}(x))\overline{(\rho(g)f)(x)}d\chi_{\partial X_0(q)^T}(x) + \int_{\partial X_0(q)^C}((\Tilde{\nabla}(\rho(g)f - \rho(g)c_p(f)))(x)\cdot \vec{n}(x))\overline{(\rho(g)f)(x)}d\chi_{\partial X_0(q)^C}(x)
\]
\[
= (A) + (B) + (C).
\]
\textbf{Remark}: here $d\chi_{\partial X_0(q)^T}(x)$ and $d\chi_{\partial X_0(q)^C}(x)$ are the corresponding volume forms on the boundaries of $X_0(q)^T$ and $X_0(q)^C$, respectively.

Since the Casimir satisfies the identity $\Delta = -\dfrac{1}{2}\left(UU^* + WW^* + \hat{H}\hat{H}^*\right)$, where the adjoint now is taken relative to the metric induced by the standard K-invariant bilinear form $(X,Y\mapsto tr(XY^*))$, and since this adjoint is simply the matrix transpose (when applied to the real elements of $\mathfrak{sl}_2(\mathbb{C})$), by the definitions of $U,W,\hat{H}$,
\[
\mathfrak{C} = -\dfrac{1}{2}\left(UU^* + WW^* + \hat{H}\hat{H}^*\right) = -\dfrac{1}{2}\left(U^2 - W^2 + \hat{H}^2\right).
\]
Hence
\[
\Tilde{\Delta} = -2\mathfrak{C} + 2W^2 = 2\Delta.
\]
Since $\sqrt{2}\rho(W) = \dfrac{d}{d\theta}$,
\[
(A) = R(\rho(g)f),
\]
we have
\[
|R(\rho(g)f) - N(\rho(g)f)| \le |(B) + (C)|.
\]
We proceed to analyze $(B)$ and $(C)$. We note that the boundaries $\partial(X_0(q)^T)$ and $\partial(X_0(q)^C)$ are equal set theoretically but have opposite normal orientations. Up to sign,
\[
(B) + (C) = \int_{\partial^T}((\Tilde{\nabla}\rho(g)c_p(f))(x)\cdot \vec{n}(x))\overline{(\rho(g)f)(x)}d\chi_{\partial^T}(x) = \int_{\partial^T}((\Tilde{\nabla}\rho(g)c_p(f))(x)\cdot \vec{n}(x))\overline{(\rho(g)c_p(f))(x)}d\chi_{\partial^T}(x),
\]
where in the last equality we used the orthogonality relations in the Iwasawa $x$-coordinate.

Splitting the set $\partial^T$ into a (disjoint) union over the cuspidal zones, i.e. $\partial^T = \cup_{\mathfrak{b}}\partial^T_{\mathfrak{b}}$, we have
\[
(B) + (C) = \sum_{\mathfrak{b}}\int_{\sigma_{\mathfrak{b}}^{-1}\partial^T_{\mathfrak{b}}}((\Tilde{\nabla}\rho(g)c_p(f))(\sigma_{\mathfrak{b}}x)\cdot \vec{n}(\sigma_{\mathfrak{b}}x))\overline{\rho(g)c_p(f)(\sigma_{\mathfrak{b}}x)}d\chi_{\partial^T}(\sigma_{\mathfrak{b}}x).
\]
where $\sigma_{\mathfrak{b}}$ is the scaling matrix sending the cusp at $\infty$ to the cusp $\mathfrak{b}$.

Since $\sigma_{\mathfrak{b}}^{-1}\partial^T_{\mathfrak{b}} = \partial^T_{\infty} = \{p\in X_0(q)|y(p) = T\}$, the set $\partial^T_{\infty}$ is naturally parametrized by the torus $\mathbb{T} = (x,\theta)\in ([0,1]/\sim)\times([0,2\pi]/\sim)$, where $[a,b]/\sim$ is the quotient of the segment $[a,b]$ by the identification $a = b$. We denote this parametrization by $g$, i.e. $g(x,\theta) = (x,T,\theta)$.

At each point $(x,\theta)\in \mathbb{T} = ([0,1]/\sim)\times([0,2\pi]/\sim)$, the tangent space, $T_{g(x,\theta)}(\partial^T_{\infty})$, is generated by the images of $\dfrac{d}{dx}$ and $\dfrac{d}{d\theta}$ under $g$, which are canonically identified with the elements of the Lie algebra corresponding to infinitesimal left translations by $N$, the unipotent group of upper triangular matrices, and infinitesimal right translation by $K = SO_2(\mathbb{R})$. Since our metric is flat, and the identification $T_p(\partial^T_{\infty})\hookrightarrow{}\mathfrak{sl}_2(\mathbb{R})$ is identical at each point $p\in \partial^T_{\infty}$ and independent of $T$, the normal, $\vec{n}(p)$, is constant, and depends only on our choice of metric. Moreover, since $T_p(\partial^T_{\infty})\hookrightarrow{}\mathfrak{sl}_2(\mathbb{R})$ is $N$-left translation and $K$-right translation invariant, the volume form, $d\chi_{\partial^T}(p)$, is too, and is therefore a constant multiple of $dxd\theta$, where the constant is independent of $x,\theta$, and hence may depend only on $y$, however, since $y = T \ll 1$ is fixed, we conclude that $d\chi_{\partial^T}(\sigma_{\mathfrak{b}}x) \ll dxd\theta$ for all $x\in \partial^T_{\infty}$.

Let $x\in \partial^T_{\infty}$. Write $\vec{n}(x) := \vec{n} = (\alpha_U,\alpha_W,\alpha_{\hat{H}})^t$, then
\[
|(B) + (C)| \ll \sum_{\mathfrak{b}}\sum_{X\in\{U,W,\hat{H}\}}\left|\int_0^1\int_0^{2\pi}((X\rho(g)c_p(f))(\sigma_{\mathfrak{b}}x))\overline{\rho(g)c_p(f)(\sigma_{\mathfrak{b}}x)}dxd\theta\right| = \sum_{\mathfrak{b}}\sum_{X\in\{U,W,\hat{H}\}}(III)_{X,\mathfrak{b}}.
\]
We note that for any pair of bases $\{A_1,A_2,A_3\}$ and $\{B_1,B_2,B_3\}$ of $\mathfrak{sl}_2(\mathbb{R})$, by the triangle inequality,
\[
\sum_{\mathfrak{b}}\sum_{X_1\in\{A_1,A_2,A_3\}}(III)_{X,\mathfrak{b}} \ll \sum_{\mathfrak{b}}\sum_{X_1\in\{B_1,B_2,B_3\}}(III)_{X,\mathfrak{b}}.
\]
Applying this observation with the basis $\{R,L,W\}$, where $L = \overline{R}$ and $R$ is defined by (see~\cite[p. 143]{BUMP}):
\[
R = e^{2i\theta}\left(iy\dfrac{d}{dx} + y\dfrac{d}{dy} + \dfrac{1}{2i}\dfrac{d}{d\theta}\right),
\]
we obtain the upper bound
\[
|(B) + (C)| \ll \sum_{\mathfrak{b}}\sum_{X_1\in\{R,L,W\}}(III)_{X,\mathfrak{b}}.
\]
Write $\rho(g)f = \sum_{n}\beta_ne_n$, then since the constant term operator and right translations commute, $\rho(g)c_p(f) = c_p(\rho(g)f) = c_p(\sum_n\beta_ne_n)$. By linearity, $\rho(g)c_p(f) = \sum_n\beta_nc_p(e_n)$. Denote by $c_n$ the constant term of $e_n$, i.e. $c_n = c_p(e_n)$, then, if $X\in \{R,L,W\}$,
\[
(III)_{X,\mathfrak{b}} = \left|\sum_n\sum_m\int_0^1\int_0^{2\pi}\beta_n\overline{\beta_m}(Xc_n)(\sigma_{\mathfrak{b}}x)\overline{c_n(\sigma_{\mathfrak{b}}x)}dxd\theta\right|.
\]
By definition (see Section~\ref{section:norm_bound}):
\[
c_n(\sigma_{\mathfrak{b}}x) = e^{in\theta}\left(\delta_{\mathfrak{a}=\mathfrak{b}}y^{1/2 + it} + \alpha(n,1/2 + it)\phi_{\mathfrak{a},\mathfrak{b}}(1/2 + it)y^{1/2 - it}\right).
\]
Since this is a constant function of $x$, we may further simplify
\[
(III)_{X,\mathfrak{b}} = \left|\sum_n\sum_m\int_0^{2\pi}\beta_n\overline{\beta_m}(Xc_n)(\sigma_{\mathfrak{b}}x)\overline{c_n(\sigma_{\mathfrak{b}}x)}d\theta\right|.
\]
Next, we note that the operators $R$ and $L$ increase a forms weight by $2$ or decrease it by $2$, respectively, and that $W$ leaves a forms' weight unchanged, by orthogonality relations in the $\theta$ coordinate, we have
\[
(III)_{R,\mathfrak{b}} = \left|\sum_n\int_0^{2\pi}\beta_{n+2}\overline{\beta_n}(Xc_{n+2})(\sigma_{\mathfrak{b}}x)\overline{c_n(\sigma_{\mathfrak{b}}x)}d\theta\right|.
\]
Recall that $|\alpha(n,1/2 + it)\phi_{\mathfrak{a},\mathfrak{b}}(1/2 + it)| \ll 1$, and $y = T \ll 1$, and therefore
\[
|c_n(\sigma_{\mathfrak{b}}x)| \ll 1\quad \forall x\in \partial^T_{\infty}.
\]
By the definition of $R$:
\[
|Re^{in\theta}y^{1/2 \pm it}| \ll (\sqrt{\lambda} + |n|)y^{1/2},
\]
and therefore,
\[
|Rc_n(\sigma_{\mathfrak{b}}x)| \ll (\sqrt{\lambda} + |n|).
\]
Plugging back to $(III)_{R,\mathfrak{b}}$, we obtain the bound
\[
(III)_{R,\mathfrak{b}} \ll \sum_n\left|\beta_{n+2}\beta_n(\sqrt{\lambda} + |n|)\right| \ll \sum_n|\beta_n|^2(\sqrt{\lambda} + |n|).
\]
A similar argument works for $(III)_{L,\mathfrak{b}}, (III)_{W,\mathfrak{b}}$. Summing over $X\in \{R,L,W\}$ and $\mathfrak{b}\in\mathfrak{A}$, we obtain the bound
\[
|(B) + (C)| \ll_{\epsilon} q^{\epsilon}\sum_n|\beta_n|^2(\sqrt{\lambda} + |n|),
\]
valid for all $\epsilon > 0$.

As a last step, we would like to bound $q^{\epsilon}\sum_n|\beta_n|^2(\sqrt{\lambda} + |n|)$ in terms of $R(\rho(g)f)$ and $P(\rho(g)f)$.

By Claim~\ref{claim:p_norm_lower_bound}, for all $n\in 2\mathbb{Z}$, $\epsilon,\delta > 0$:
\[
q^{\epsilon}(\sqrt{\lambda} + |n|)^{\delta}P(e_n) \gg_{\epsilon,\delta} 1.
\]
Hence,
\[
|(B) + (C)| \ll_{\epsilon,\delta} q^{\epsilon}\sum_n|\beta_n|^2(\sqrt{\lambda} + |n|)^{1+\delta}P(e_n).
\]
Writing 
\[
|\beta_n|^2(\sqrt{\lambda} + |n|)^{1+\delta}P(e_n) = 
|\beta_n|^{1-\delta}P(e_n)^{(1-\delta)/2}(\sqrt{\lambda} + |n|)^{1+\delta}|\beta_n|^{1 + \delta}P(e_n)^{(1+\delta)/2},
\]
and applying the H\"{o}lder inequality, valid under the assumption that $0 < \delta < 1/2$, we find that
\[
\sum_n|\beta_n|^2(\sqrt{\lambda} + |n|)^{1+\delta}P(e_n) \le \left(\sum_n|\beta_n|^2P(e_n)\right)^{(1-\delta)/2}\left(\sum_n|\beta_n|^2(\sqrt{\lambda} + |n|)^2
P(e_n)\right)^{(1+\delta)/2}
\]
\[
\ll
\left(\sum_n|\beta_n|^2P(e_n)\right)^{(1-\delta)/2}\left(\sum_n|\beta_n|^2(2\lambda + n^2)
P(e_n)\right)^{(1+\delta)/2} = 
P(\rho(g)f)^{(1-\delta)/2}R(\rho(g)f)^{(1+\delta)/2},
\]
which implies that for all $\epsilon > 0$ and $1/2 > \delta > 0$:
\[
\left|R(\rho(g)f) - N(\rho(g)f)\right| \ll_{\epsilon,\delta} q^{\epsilon}P(\rho(g)f)^{(1-\delta)/2}R(\rho(g)f)^{(1+\delta)/2}.
\]
Our claim follows from substituting $\delta = 1/4$.
\end{proof}

\begin{claim}
\label{claim:p_norm_lower_bound}
The following bound holds,
\[
P(e_n) \gg
\begin{cases}
1& |t| \ll \left(\log^2(1 + q) + \log(1 + |n/2|)\right)^{-1}\\
t^2& \left(\log^2(1 + q) + \log(1 + |n/2|)\right)^{-1} \ll |t| \le 1,\\
1& |t| > 1.
\end{cases}
\]
\end{claim}
\begin{proof}
Clearly,
\[
P(e_n) = \int_{X_0(q)}|\Lambda^Te_n(x)|^2d\chi(x) \ge \int_{X_0(q)^T\setminus X_0(q)^1}|\Lambda^Te_n(x)|^2d\chi(x),
\]
where here the notation $X_0(q)^L$ denotes the set of all points $x\in X_0(q)$ with height at most $L$. The advantage of the domain $X_0(q)^T\setminus X_0(q)^1$ is that it is rectangular so that we may use orthogonality relations, and that $\Lambda^Te_n = e_n$ in this domain, hence
\[
P(e_n) \ge \int_{X_0(q)^T\setminus X_0(q)^1}|e_n(x)|^2d\chi(x).
\]
Next, denote by $\sigma_{\mathfrak{b}}$ the scaling matrix sending the cusp at $\infty$ to the cusp at $\mathfrak{b}$. We write $X_0(q)^T\setminus X_0(q)^1 = \cup_{\mathfrak{b}}R_{\mathfrak{b}}$, where the set $R_{\mathfrak{b}}$ denotes the set of points in the cuspidal zone of the cusp $\mathfrak{b}$ of height between $1$ and $T$. Then,
\[
\int_{X_0(q)^T\setminus X_0(q)^1}|e_n(x)|^2d\chi(x) = \sum_{\mathfrak{b}}\int_{\sigma_{\mathfrak{b}}^{-1}R_{\mathfrak{b}}}|e_n(\sigma_{\mathfrak{b}}x)|^2d\chi(\sigma_{\mathfrak{b}}x) \ge 
\sum_{\mathfrak{b}}\int_{\sigma_{\mathfrak{b}}^{-1}R_{\mathfrak{b}}}|c_n(\sigma_{\mathfrak{b}}x)|^2d\chi(\sigma_{\mathfrak{b}}x),
\]
where $c_n$ denotes the constant term of $e_n$, given by the formula
\[
c_n(\sigma_{\mathfrak{b}}x) = e^{in\theta}\left(\delta_{\mathfrak{a}=\mathfrak{b}}y^{1/2 + it} + \alpha(n,1/2 + it)\phi_{\mathfrak{a},\mathfrak{b}}(1/2 + it)y^{1/2 - it}\right).
\]
Since $\sigma_{\mathfrak{b}}^{-1}R_{\mathfrak{b}} = R_{\infty}$ for all $\mathfrak{b}$, we obtain the lower bound
\[
P(e_n) \ge 
\sum_{\mathfrak{b}}\int_0^1\int_1^T\int_0^{2\pi}|e^{in\theta}\left(\delta_{\mathfrak{a}=\mathfrak{b}}y^{1/2 + it} + \alpha(n,1/2 + it)\phi_{\mathfrak{a},\mathfrak{b}}(1/2 + it)y^{1/2 - it}\right)|^2\dfrac{dxdyd\theta}{y^2}
\]
\[
\gg \int_1^T|1 + y^{-2it}\alpha(n,1/2 + it)\phi_{\mathfrak{a},\mathfrak{a}}(1/2 + it)|^2\dfrac{dy}{y} + \left(1 - |\phi_{\mathfrak{a},\mathfrak{a}}(1/2 + it)|\right).
\]
Where we used the fact that $\Phi(1/2 + it)$ is unitary (see Proposition~\ref{proposition:scattering_matrix}) and $|\alpha(n,1/2 + it)| = 1$ (see Claim~\ref{claim:truncated_eisenstein_norm_bound}) and by definition $T \ge e^{2\pi}$.

As in the statement of Claim~\ref{claim:p_norm_lower_bound}, we split our analysis into cases with respect to how large $|t|$ is. 

We have three cases:
\begin{itemize}
    \item $|t| \ll \left(\log^2(1 + q) + \log(1 + |n/2|)\right)^{-1}$,
    \item $\left(\log^2(1 + q) + \log(1 + |n/2|)\right)^{-1} \ll |t| \le 1$,
    \item $|t| > 1$.
\end{itemize}

\begin{claim}
Assume $|t| \ll \left(\log^2(1 + q) + \log(1 + |n/2|)\right)^{-1}$, then
\[
P(e_n) \gg 1.
\]
\end{claim}
\begin{proof}
Arguing as in the proof of Claim~\ref{claim:third_term}, assuming $y\in [1,T]$, we have
\[
\left|\dfrac{\alpha(n,1/2 + it)\phi_{\mathfrak{a},\mathfrak{a}}(1/2 + it)y^{-2it} - 1}{t}\right| \ll \log^2(1 + q) + \log(1 + |n/2|).
\]
Therefore, since
\[
\left|\left|1 + \alpha(n,1/2 + it)\phi_{\mathfrak{a},\mathfrak{a}}(1/2 + it)y^{-2it}\right| - 2\right| \le \left|\alpha(n,1/2 + it)\phi_{\mathfrak{a},\mathfrak{a}}(1/2 + it)y^{-2it} - 1\right|
\]
\[
\ll |t|\left(\log^2(1 + q) + \log(1 + |n/2|)\right),
\]
we find that there exists some absolute constant $C > 1$, such that for $|t| < \dfrac{1}{C\left(\log^2(1 + q) + \log(1 + |n/2|)\right)}$, one has
\[
\left|1 + \alpha(n,1/2 + it)\phi_{\mathfrak{a},\mathfrak{a}}(1/2 + it)y^{-2it}\right| \ge 1.
\]
Thus,
\[
P(e_n) \gg \int_1^T|1 + y^{-2it}\alpha(n,1/2 + it)\phi_{\mathfrak{a},\mathfrak{a}}(1/2 + it)|^2\dfrac{dy}{y} \ge \log T \ge 1,
\]
which completes the proof.
\end{proof}
From this point onward, we split the analysis into two cases. If $|\phi_{\mathfrak{a},\mathfrak{a}}(1/2 + it)| < 1/2$, then clearly
\[
P(e_n)\gg \left(1 - |\phi_{\mathfrak{a},\mathfrak{a}}(1/2 + it)|\right) > 1/2
\]
and Claim~\ref{claim:p_norm_lower_bound} holds. 

Otherwise,
\begin{claim}
Assume $\left(\log^2(1 + q) + \log(1 + |n/2|)\right)^{-1} \ll |t| \le 1$ and that $|\phi_{\mathfrak{a},\mathfrak{a}}(1/2 + it)| \ge 1/2$, then
\[
P(e_n) \gg t^2.
\]
\end{claim}
\begin{proof}
By claim~\ref{claim:integrand_sup_bound}, since $T \ge e^{2\pi} \ge e$ by assumption,
\[
\sup_{y\in [1,T]}|1 + y^{-2it}\alpha(n,1/2 + it)\phi_{\mathfrak{a},\mathfrak{a}}(1/2 + it)| \gg |t|.
\]
Combining with Claim~\ref{claim:s_and_I_relations}, we find that
\[
P(e_n) \ge \int_1^T|1 + y^{-2it}\alpha(n,1/2 + it)\phi_{\mathfrak{a},\mathfrak{a}}(1/2 + it)|^2\dfrac{dy}{y} \gg 
t^2,
\]
which completes the proof.
\end{proof}

\begin{claim}
Assume $|t| > 1$ and that $|\phi_{\mathfrak{a},\mathfrak{a}}(1/2 + it)| \ge 1/2$, then
\[
P(e_n) \gg 1.
\]
\end{claim}
\begin{proof}
Our Claim follows from Claim~\ref{claim:large_t_bound}.
\end{proof}

Combining the previous bounds we find that
\[
P(e_n) \gg
\begin{cases}
1& |t| \ll \left(\log^2(1 + q) + \log(1 + |n/2|)\right)^{-1}\\
t^2& \left(\log^2(1 + q) + \log(1 + |n/2|)\right)^{-1} \ll |t| \le 1,\\
1& |t| > 1.
\end{cases}
\]
Thus completing the proof.
\end{proof}

\section{Bounding the $Q$-right shift norm}
Our goal in this section is to prove:
\begin{claim}
\label{claim:real_q_norm_right_shift_bound}
For all small translates $g\in R(\delta)$ (see Claim~\ref{claim:delta_existence}), one has
\[
\forall g\in R(\delta): ||\rho(g)||_Q \le ||\rho(g)||_P^{1/2 - \epsilon_0}||\rho(g)||_R^{1/2 + \epsilon_0}.
\]
\end{claim}
\textbf{Remark}: recall that $0 < \epsilon_0 < 1/2$ is as in Definition~\ref{definition:q_norm}.
\begin{proof}
Quadratic interpolation (see~\cite{QUAD1} and~\cite[Chapter 4, Theorem 1.13]{KREIN} for reference), is the study of the interpolation of a pair of Banach spaces in the special case where both Banach spaces are Hilbert spaces, and their interpolation is also a Hilbert space.

Although our ``interpolated norm" $Q$ was defined independently, and had nothing to do with quadratic interpolation a priori,~\cite[Theorem 3.3]{QUAD1} shows that the norm $Q$ is, in fact, the norm of an interpolating Hilbert space, and~\cite[Theorem 3.5]{QUAD1} shows that this interpolation is exact of exponent $\theta_0 = 1/2 + \epsilon_0$.

Having the $Q$-norm being an exact interpolation of exponent $\theta_0$ means that for all linear operators $A:V\longrightarrow V$, for which $||A||_P,||A||_R < \infty$, i.e. the $P$-norm and the $R$-norm of $A$ are bounded, the $Q$-norm of $A$ is also bounded, and moreover,
\[
||A||_Q \le ||A||_P^{1 - \theta_0}||A||_R^{\theta_0}.
\]
This is the interpolation inequality. Applying the interpolation inequality to the operator $\rho(g)$ for $g\in R(\delta)$, i.e. $g$ a small translate, we know that $||\rho(g)||_P,||\rho(g)||_R < \infty$ and the claim follows.
\end{proof}

\appendix
\section{The approximate functional equation technique}
\label{appendix:blomer}
We sketch an alternative approach to our theorem, which was explained to us by Valentin Blomer~\cite{BLOMER}.

\begin{theorem}
Let $g\in SL_2(\mathbb{R})$ have Iwasawa coordinates $x,y,\theta$, let $q$ be a positive squarefree integer, $n\in 2\mathbb{Z}$, and $\mathfrak{a}$ be a cusp of $X_0(q)$. For $t\in\mathbb{R}$, denote by $E_{\mathfrak{a},n}(*,1/2 + it)$ the Eisenstein series of the cusp $\mathfrak{a}$, weight $n$, with spectral parameter $1/2 + it$, normalized by the constant term. Then
\[
|E_{\mathfrak{a},n}(x,1/2 + it)| \ll_{\epsilon} q^{\epsilon}(y^{1/2} + y^{-1/2})(1 + |t| + |n|)^{1/2+{\epsilon}},\quad \forall 0 < \epsilon < 1/2.
\]
\end{theorem}
\begin{proof}
The starting point is a reduction from the squarefree level $q$ case to the level 1 case, via the following formula~\cite{CONREYIWANIEC}[3.25]
\[
E_{\mathfrak{a},0}(z,1/2 + it) = 
\prod_{p|q}(1-\dfrac{1}{p^{1 + 2it}})^{-1}\dfrac{\mu(v)}{(qv)^{1/2 + it}}\sum_{\beta | v}\sum_{\gamma | w}\dfrac{\mu(\beta \gamma)\beta^{1/2  + it}}{\gamma^{1/2 + it}}E_0(\beta\gamma z, 1/2 + it).
\]
In this formula, $\mathfrak{a} = v$ is some divisor of $q$, and $w$ is the cusp's width. Together, the two integers satisfy the equality $vw = q$.

The formula gives us a way to write a spherical Eisenstein series of level $q$ (the left hand side) as a linear combination of Eisenstein series of level 1, possibly with their argument multiplied by some integer, which is a divisor of $q$ (the right hand side).

Since we are interested in bounding Eisenstein series of arbitrary weight, we may adjust the weight of the formula by applying the raising/lowering operators on both sides, to obtain
\[
E_{\mathfrak{a},n}(x,y,\theta,1/2 + it) = 
\prod_{p|q}(1-\dfrac{1}{p^{1 + 2it}})^{-1}\dfrac{\mu(v)}{(qv)^{1/2 + it}}\sum_{\beta | v}\sum_{\gamma | w}\dfrac{\mu(\beta \gamma)\beta^{1/2 + it}}{\gamma^{1/2 + it}}E_n(\beta\gamma x, \beta\gamma y, \theta, 1/2 + it).
\]
To see that this formula indeed holds, note that in $SL_2(\mathbb{R})$ coordinates, $\beta\gamma z$ can be written as
\[
\begin{pmatrix}
(\beta\gamma)^{1/2}& 0\\
0& (\beta\gamma)^{-1/2}
\end{pmatrix}
\begin{pmatrix}
1& x\\
0& 1
\end{pmatrix}
\begin{pmatrix}
y^{1/2}& 0\\
0& y^{-1/2}
\end{pmatrix},
\]
i.e. multiplication by $\beta\gamma$ can be interpreted as a certain left translation. Since left and right translations commute, and since the raising operator is a limit of right translations (as a Lie derivative), it follows that the two translations commute. Therefore
\[
R^{n/2}E_0(\beta\gamma z, 1/2 + it) = E_{n}(\beta\gamma x, \beta\gamma y, \theta, 1/2 + it),
\]
and similarly for the lowering operator.

Let $\epsilon > 0$ be arbitrary. Applying the triangle inequality, we obtain the bound
\[
|E_{\mathfrak{a},n}(x,y,\theta,1/2 + it)| \ll_{\epsilon} 
q^{\epsilon - 1/2}\max_{d|q}|E_n(d x, d y, \theta, 1/2 + it)|.
\]
To bound $|E_n(d x, d y, \theta, 1/2 + it)|$, recall the functional equation, satisfied by $\Tilde{E_n}(x,y,\theta,1/2 + it) = E_n(x, y, \theta, 1/2 + it)\zeta(1 + 2it)$ (see~\cite{FUNCTIONAL_EQUATION}[4.47, 4.48]),
\[
\Gamma(1/2 + it + |n|/2)\Tilde{E_n}(x,y,\theta,1/2 + it) = \Gamma(1/2 - it + |n|/2)\Tilde{E_n}(x,y,\theta,1/2 - it).
\]
The weight $n$ Eisenstein series is a generalized Dirichlet series, given by the formula
\[
\Tilde{E_n}(x,y,\theta,s) = e^{in\theta}y^s\sum_{(c,d)\neq (0,0)}\left(\dfrac{c\overline{z} + d}{|cz + d|}\right)^n\dfrac{1}{|cz + d|^{2s}},
\]
where $z = x + iy$. This formula is obtained by applying the raising/lowering operators term by term over $\Tilde{E_0}(x,y,\theta,s) = y^s\sum_{(c,d)\neq (0,0)}\dfrac{1}{|cz + d|^{2s}}$.

Using the approximate functional equation, as in~\cite{BLOMER_APPROX}[2.1] yields the following bound on $|\Tilde{E_n}(x,y,\theta,1/2 + it)|$:
\[
|\Tilde{E_n}(x,y,\theta,1/2 + it)| \ll_{\epsilon} 1 + y^{1/2}\sum_{\substack{(c,d)\neq (0,0)\\|cz + d|^2 \ll (1 + |t| + |n|)^{1+{\epsilon}}}}\dfrac{1}{|cz + d|}.
\]
Note that there are at most $O((1 + \sqrt{T}/y)\sqrt{T})$ choices for $(c, d)$ with $|cz + d|^2 \ll T$, and therefore
\[
|\Tilde{E_n}(x,y,\theta,1/2 + it)| \ll_{\epsilon} 1 + y^{1/2}\sum_{k: 2^k \ll (1 + |t| + |n|)^{1+{\epsilon}}}\sum_{\substack{(c,d)\neq (0,0)\\2^{k-1}\le |cz + d|^2 < 2^k}}\dfrac{1}{|cz + d|} \ll_{\epsilon} 1 + y^{1/2}\sum_{k: 2^k \ll (1 + |t| + |n|)^{1+{\epsilon}}}(1 + 2^{k/2}/y)
\]
\[
\ll_{\epsilon} 1 + y^{1/2}(1 + |t| + |n|)^{\epsilon} + y^{-1/2}(1 + |t| + |n|)^{1/2+{\epsilon}/2}.
\]

Recalling that $|\zeta(1 + 2it)| \gg 1/(1 + |t|)^{\epsilon}$ for $t\in \mathbb{R}$, and that $E_n(x,y,\theta,1/2 + it) = \Tilde{E_n}(x,y,\theta,1/2 + it)/\zeta(1 + 2it)$, we deduce that
\[
|E_n(x,y,\theta,1/2 + it)| \ll_{\epsilon} (y^{1/2} + y^{-1/2})(1 + |t| + |n|)^{1/2+{\epsilon}}.
\]
Plugging this into $|E_{\mathfrak{a},n}(x,y,\theta,1/2 + it)| \ll_{\epsilon} 
q^{\epsilon - 1/2}\max_{d|q}|E_n(d x, d y, \theta, 1/2 + it)|$, the theorem follows.
\end{proof}

\printbibliography

@article{IWASAR,
  title={L-$\infty$ norms of eigenfunctions of arithmetic surfaces},
  author={Iwaniec, Henryk and Sarnak, Peter},
  journal={Annals of Mathematics},
  pages={301--320},
  year={1995},
  publisher={JSTOR}
}

@article{BLOMERHOLO,
  title={Bounding sup-norms of cusp forms of large level},
  author={Blomer, Valentin and Holowinsky, Roman},
  journal={Inventiones mathematicae},
  volume={179},
  number={3},
  pages={645--681},
  year={2010},
  publisher={Springer}
}

@article{TEMP1,
  title={On the sup-norm of Maass cusp forms of large level},
  author={Templier, Nicolas},
  journal={Selecta Mathematica},
  volume={16},
  number={3},
  pages={501--531},
  year={2010},
  publisher={Springer}
}

@article{TEMP-HARC1,
  title={On the sup-norm of Maass cusp forms of large level: II},
  author={Harcos, Gergely and Templier, Nicolas},
  journal={International Mathematics Research Notices},
  volume={2012},
  number={20},
  pages={4764--4774},
  year={2012},
  publisher={OUP}
}

@article{TEMP-HARC2,
  title={On the sup-norm of Maass cusp forms of large level. III},
  author={Harcos, Gergely and Templier, Nicolas},
  journal={Mathematische Annalen},
  volume={356},
  number={1},
  pages={209--216},
  year={2013},
  publisher={Springer}
}

@article{TEMP2,
  title={Hybrid sup-norm bounds for Hecke--Maass cusp forms},
  author={Templier, Nicolas},
  journal={Journal of the European Mathematical Society},
  volume={17},
  number={8},
  pages={2069--2082},
  year={2015}
}

@article{YOUNG,
  title={A note on the sup norm of Eisenstein series},
  author={Young, Matthew P},
  journal={The Quarterly Journal of Mathematics},
  volume={69},
  number={4},
  pages={1151--1161},
  year={2018},
  publisher={Oxford University Press}
}

@inproceedings{HUANGXU,
  title={Sup-norm bounds for Eisenstein series},
  author={Huang, Bingrong and Xu, Zhao},
  booktitle={Forum Mathematicum},
  volume={29},
  number={6},
  pages={1355--1369},
  year={2017},
  organization={De Gruyter}
}

@article{BERREZ,
  title={Sobolev norms of automorphic functionals},
  author={Bernstein, Joseph and Reznikov, Andre},
  journal={International Mathematics Research Notices},
  volume={2002},
  number={40},
  pages={2155--2174},
  year={2002},
  publisher={OUP}
}

@article{MUSICZEHAVI,
  title={Sectorial equidistribution of the roots of $X^2 + 1$ modulo primes},
  author={Musicantov, Evgeny and Zehavi, Sa'ar},
  journal={arXiv preprint arXiv:2112.07494},
  year={2021}
}

@book{KUBOTA,
  title={Elementary theory of Eisenstein series},
  author={Kubota, Tomio},
  year={1973},
  publisher={Halsted Press}
}

@article{CCF,
  title={Transmission eigenvalues and the Riemann zeta function in scattering theory for automorphic forms on Fuchsian groups of type I},
  author={Cakoni, Fioralba and Chanillo, Sagun},
  journal={Acta Mathematica Sinica, English Series},
  volume={35},
  number={6},
  pages={987--1010},
  year={2019},
  publisher={Springer}
}

@book{BUMP,
  title={Automorphic forms and representations},
  author={Bump, Daniel},
  number={55},
  year={1998},
  publisher={Cambridge university press}
}

@misc{MORA,
  title={Letter to Morawetz},
  author={Sarnak, Peter},
  year={2004}
}

@article{RUDNICK,
  title={On the asymptotic distribution of zeros of modular forms},
  author={Rudnick, Ze{\'e}v},
  journal={International Mathematics Research Notices},
  volume={2005},
  number={34},
  pages={2059--2074},
  year={2005},
  publisher={OUP}
}

@article{NODAL,
  title={Nodal domains of Maass forms I},
  author={Ghosh, Amit and Reznikov, Andre and Sarnak, Peter},
  journal={Geometric and Functional Analysis},
  volume={23},
  number={5},
  pages={1515--1568},
  year={2013},
  publisher={Springer}
}

@article{JORGENSON,
  title={Bounding the sup-norm of automorphic forms},
  author={Jorgenson, Jay and Kramer, Jurg},
  journal={Geometric \& Functional Analysis GAFA},
  volume={14},
  number={6},
  pages={1267--1277},
  year={2004},
  publisher={Springer}
}

@article{QUAD1,
  title={Interpolation of Hilbert and Sobolev spaces: quantitative estimates and counterexamples},
  author={Chandler-Wilde, Simon N and Hewett, David P and Moiola, Andrea},
  journal={Mathematika},
  volume={61},
  number={2},
  pages={414--443},
  year={2015},
  publisher={London Mathematical Society}
}

@book{KREIN,
  title={Interpolation of linear operators},
  author={Krein, Selim Grigorevich and Semenov, EM},
  volume={54},
  year={2002},
  publisher={American Mathematical Soc.}
}

@book{IWANIECBOOK,
  title={Spectral methods of automorphic forms},
  author={Iwaniec, Henryk},
  volume={53},
  year={2002},
  publisher={American Mathematical Soc.}
}

@article{CONREYIWANIEC,
  title={The cubic moment of central values of automorphic L-functions},
  author={Conrey, J Brian and Iwaniec, Henryk},
  journal={Annals of mathematics},
  volume={151},
  number={3},
  pages={1175--1216},
  year={2000},
  publisher={JSTOR}
}

@article{FUNCTIONAL_EQUATION,
  title={The subconvexity problem for Artin L--functions},
  author={Duke, William and Friedlander, John B and Iwaniec, Henryk},
  journal={Inventiones mathematicae},
  volume={149},
  number={3},
  pages={489--577},
  year={2002},
  publisher={Springer-Verlag}
}

@article{BLOMER_APPROX,
  title={Epstein zeta-functions, subconvexity, and the purity conjecture},
  author={Blomer, Valentin},
  journal={Journal of the Institute of Mathematics of Jussieu},
  volume={19},
  number={2},
  pages={581--596},
  year={2020},
  publisher={Cambridge University Press}
}

@article{BLOMER,
  title={Private communication},
  author={Blomer, Valentin},
  year={2022}
}
\end{document}